\documentclass{article}

\usepackage{amsmath, amssymb, amsthm, mathrsfs}
\usepackage{indentfirst}
\usepackage{graphicx, tikz, subfigure, booktabs}
\usepackage{geometry}
\usepackage{caption}

\usepackage{hyperref}

\usepackage{cleveref}
\crefname{section}{Section}{Sections}
\crefname{figure}{Figure}{Figures}
\crefname{table}{Table}{Tables}
\crefname{equation}{}{}
\crefname{theorem}{Theorem}{Theorems}
\crefname{lemma}{Lemma}{Lemmas}
\crefname{remark}{Remark}{Remarks}
\crefname{problem}{Problem}{Problems}

\newtheorem{theorem}{Theorem}[section]
\newtheorem{example}{Example}[section]
\newtheorem{problem}{Problem}[section]
\newtheorem{remark}{Remark}[section]

\newtheorem{definition}{Definition}[section]

\graphicspath{{Figures/}}

\begin{document}
	
	\title{Imaging an acoustic obstacle and its excitation sources from phaseless near-field data}
	
	\author{
		Deyue Zhang\thanks{School of Mathematics, Jilin University, Changchun, China, {\it dyzhang@jlu.edu.cn}},
		Yue Wu\thanks{School of Mathematics, Jilin University, Changchun, China, {\it wuy20@mails.jlu.edu.cn}}, 
		\ and 
		Yukun Guo\thanks{School of Mathematics, Harbin Institute of Technology, Harbin, China. {\it ykguo@hit.edu.cn} (Corresponding author)}
	}
	\date{}
	
\maketitle
	
\begin{abstract}
	This paper is concerned with reconstructing an acoustic obstacle and its excitation sources from the phaseless near-field measurements. By supplementing some artificial sources to the inverse scattering system, this co-inversion problem can be decoupled into two inverse problems: an inverse obstacle scattering problem and an inverse source problem, and the corresponding uniqueness can be established. This novel decoupling technique requires some extra data but brings in several salient benefits. First, our method is fast and easy to implement. Second, the boundary condition of the obstacle is not needed. Finally, this approximate decoupling method can be applied to other co-inversion problems, such as determining the medium and its excitation sources. Several numerical examples are presented to demonstrate the feasibility and effectiveness of the proposed method.
\end{abstract}


%=============================================================

\section{Introduction}

Inverse problems of determining unknown excitations or scattering inclusions by acoustic waves appear in diverse areas of scientific and engineering importance, e.g.,  nondestructive evaluation, sonar detection, and ultrasonic tomography. In many areas of applied sciences, it is very difficult or extremely expensive to access the phased data of the complex-valued time-harmonic wave field, while the phaseless/modulus data is usually much cheaper to acquire \cite{KR17, MH17, Rom20}. Hence, reconstruction from phaseless data is both theoretically and practically important and has received considerable attention in the literature in recent years \cite{ACZ16, DLL20, LL15, LLW17, Nov16, XZZ20, ZGSL20, ZCLL17}. Recently, there is also a surge of interest in the co-inversion problem of simultaneously recovering the source term and the passive inhomogeneity from the scattering measurements, see, e.g. \cite{BLT21, CG22, HKZ20, LLM19, LLM21}. However, to our knowledge, the simultaneous reconstruction of source and scatterer from phaseless data has not yet been studied in the existing literature.

In this paper, we investigate a new model of simultaneously reconstructing the point sources and the impenetrable obstacle from phaseless scattering data. This phaseless co-inversion problem suffers from the drastic difficulties of nonlinearity and severe ill-posedness, in particular, the dual unknowns and the loss of phase information substantially obstruct the application of algorithms for phased inverse scattering problems. Meanwhile, it is well acknowledged that a lack of information cannot be remedied by merely the mathematical implementation itself and therefore the incorporation of additional information is indispensable for achieving effective reconstruction. To this end, motivated by the techniques of artificial reference sources/objects for phaseless inverse scattering problems \cite{DZG19, JLZ19a, JLZ19b, ZGLL18, SZG19, ZWGL20}, as well as the recent decoupling strategy for the interior co-inversion problem \cite{ZGWC22}, in the current work we propose a novel method for tackling the challenging phaseless co-inversion problem. The basic idea of our method is to decode the source and scatterer components from the scattering system with the help of the additional illuminations due to the so-called reference sources. These artificially appended sources play a significant role in decoupling the co-inversion problem into the usual inverse source subproblem and inverse obstacle scattering subproblems. Then the uniqueness issue on these subproblems could be separately analyzed, namely, the shape of the scatterer together with the boundary condition could be uniquely determined, meanwhile, the unknown source locations can be also uniquely identified. Moreover, the decoupled subproblems can be numerically treated easily by the direct sampling scheme and the reverse time migration method, respectively. The overall algorithm works with the data due to only a single frequency and does not rely on any solver of the forward problem. Since the computational demanding process is not involved, the proposed method can be implemented easily. In addition, neither the number of obstacles nor the boundary conditions are required to be priorly known. In our view, the aforementioned features of our study constitute the noteworthy novelty of this article.

The remaining part of this article is organized as follows. In the next section, we introduce the mathematical formulation of the phaseless co-inversion problem as well as the incorporation of reference sources. In \cref{sec: uniqueness}, the theoretical results are provided to justify the unique identifiability of underlying sources and obstacles from the modulus of near-field measurements. In \cref{sec:algorithms}, the phaseless co-inversion problem is decoupled into two subproblems of inverse obstacle scattering and inverse source problem. Then the subproblems are tackled respectively with direct imaging indicators. The indicating behavior for locating the sources is analyzed as well. Numerical experiments are presented in \cref{sec:examples} to confirm the theoretical analysis of our inversion process and illustrate the effectiveness and robustness of our newly proposed phaseless co-inversion algorithm.

\section{Problem setting}\label{sec:OP}

We begin this section by introducing the phaseless co-inversion problem of an obstacle and its excitation point sources under consideration. Let $D\subset\mathbb{R}^2$ be a simply connected bounded domain with $C^2$ boundary $\partial D$. For a generic point $z\in \mathbb{R}^2\backslash\overline{D}$, the incident field $u^i$ due to the point source located at $z$ is given by
\begin{equation}\label{pointsource}
	u^i (x; z)=\frac{\mathrm{i}}{4}H_0^{(1)}(k|x-z|), \quad x\in \mathbb{R}^2\backslash (D \cup\{z\}),
\end{equation}
where $H_0^{(1)}$ is the Hankel function of the first kind of order zero, and $k>0$ is the wavenumber. Then, the forward scattering problem can be stated as follows: given the source point $z$ and the obstacle $D$, find the scattered field $u^s(x; z)$ which satisfies the following boundary value problem (see \cite{CK19}):
\begin{align}
	\Delta u^s+ k^2 u^s & = 0\quad \mathrm{in}\ \mathbb{R}^2\backslash\overline{D},\label{eq:Helmholtz} \\
	\mathscr{B}u & = 0 \quad \mathrm{on}\ \partial D, \label{eq:boundary_condition} \\
	\lim\limits_{r:=|x|\to\infty} \sqrt{r}\bigg(&\frac{\partial u^s}{\partial r} - \mathrm{i} k u^s\bigg)=0, \label{eq:Sommerfeld}
\end{align}
where $u(x; z)=u^i(x; z)+u^s(x; z)$ denotes the total field and \cref{eq:Sommerfeld} is the Sommerfeld radiation condition. Here $\mathscr{B}$ in \cref{eq:boundary_condition} is the boundary operator defined by
\begin{equation}\label{BC}
	\mathscr{B}u=
	\begin{cases}
		u, & \text{for a sound-soft obstacle},  \\
		\dfrac{\partial u}{\partial \nu}+ \mathrm{i} k\lambda u, & \text{for an impedance obstacle}.
	\end{cases}
\end{equation}
where $\nu$ is the unit outward normal to $\partial D$ and $\lambda$ is a real parameter. This boundary condition \cref{BC} covers the Dirichlet/sound-soft boundary condition, the Neumann/sound-hard boundary condition ($\lambda=0$), and the impedance boundary condition ($\lambda\neq 0$). It is well known that the forward scattering problem \eqref{eq:Helmholtz}-\eqref{eq:Sommerfeld} admits a unique solution $u^s\in H_{\rm loc}^1(\mathbb{R}^2\backslash\overline{D})$ (see, e.g., \cite{CK19, McLean}),  and the scattered wave $u^s$ has the following asymptotic behavior
$$
	u^s(x; z)=\frac{\mathrm{e}^{\mathrm{i} k|x|}}{\sqrt{|x|}}\left\{ u^{\infty}(\hat{x}; z)+\mathcal{O}\left(\frac{1}{|x|}\right) \right\}, \quad |x|\to\infty
$$
uniformly in all observation directions $\hat{x}=x/|x|\in\mathbb{S}^1:=\{x\in\mathbb{R}^2: |x|=1\}$. The complex-valued analytic function $u^{\infty}(\hat{x}; z)$ defined on the unit circle $\mathbb{S}^1$ is called the far field pattern or scattering amplitude (see \cite{CK19}).

Let $z_j\in \mathbb{R}^2\backslash\overline{D}\, (j=1,\cdots, N)$ be mutually distinct source points, $P:=\cup_{j=1}^{N}\{z_j\}$, $B_\rho=\{x\in  \mathbb{R}^2:|x|<\rho\}$ and $\Gamma_\rho=\partial B_\rho$, such that $\left(\overline{D}\cup P\right)\subset B_\rho$.
Assume that $\Gamma_R=\partial  B_R$ is the measurement curve with $B_R=\{x\in  \mathbb{R}^2:|x|<R\}$ and $R>\rho$. Denote the superposition of wave fields by
$$
w(x; P)=\sum_{j=1}^N w(x; z_j),\quad z_j\in P,\ j=1, \cdots, N,
$$
with $w$ in place of $u^i, u^s, u$ or $u^\infty$, accordingly. Then, the phaseless co-inversion problem under consideration is to determine the obstacle-source pair $(\partial D, P)$ from the phaseless measurements $\{|u(x; P)|: x\in \Gamma_R\}$.

To our knowledge, this model inverse problem as well as its uniqueness issue were not studied in the literature. Motivated by the reference point techniques for tackling the phaseless inverse scattering problem, we will consider incorporating artificial point sources into the co-inversion system. Let $z_1, z_2\in \mathbb{R}^2\backslash(\overline{D}\cup P)$ be two additional source points and $t_1, t_2\in\mathbb{R}$. By the linearity of the direct scattering problem, the total field produced by the incident waves $u^i(x; P)+t_1 u^i(x; z_1)+t_2 u^i(x; z_2)$ is given by
$$
	u(x; P, t_1 z_1, t_2 z_2):= u(x; P)+t_1 u(x; z_1)+t_2 u(x; z_2), \ x\in \mathbb{R}^2\backslash(\overline{D}\cup P \cup\{z_1, z_2\}),
$$

To introduce the phaseless co-inversion problem with artificial sources, we also need the following definition of the admissible curve.

\begin{definition}[\cite{ZGSL20}] \label{def:admissible_curve}
	(Admissible curve) An open curve $\Gamma$ is called an admissible curve with respect to domain $\Omega$ if
\begin{itemize}
	\item[\rm (i)] $\Omega$ is simply-connected;
	\item [\rm (ii)] $\partial \Omega$ is analytic homeomorphic to $\mathbb{S}^1$;
	\item[\rm (iii)] $k^2$ is not a Dirichlet eigenvalue of $-\Delta$ in $\Omega$;
	\item[\rm (iv)] $\Gamma \subset \partial \Omega$ is a one-dimensional analytic manifold with nonvanishing measure.
\end{itemize}
\end{definition}

%\begin{remark}
%	We would like to point out that this requirement for the admissibility of $\Gamma$ is quite mild and thus can be easily fulfilled. For instance, $\Omega$ can be chosen as a disk whose radius is less than $2.4048/k$ and  $\Gamma$ is chosen as an arbitrary half circle.
%\end{remark}

Now, the phaseless co-inversion problem with artificial sources under consideration can be stated as the following.

\begin{problem}[Phaseless co-inversion problem]\label{prob:co-inversion}
	Let $D$ be the impenetrable obstacle with boundary condition $\mathscr{B}$ and $P$ be a set of distinct source points such that $(D\cup P)\subset B_\rho\subset B_R$. Assume that $\Gamma$ is an admissible curve with respect to $G$ with $G\subset B_R\backslash \overline{B_\rho}$. Given the phaseless near-field data
	\begin{align*}
		& \{|u(x; P)|: x\in \Gamma_R\},  \\
		& \{|u(x; P, t_{\ell} z_0)|: x\in \Gamma_R,  \ell=1,2\}, \\
		& \{|u(x; P, t_{\ell} z)|: x\in \Gamma_R, z\in \Gamma_\rho\cup\Gamma, \ell=1,2\}, \\
		& \{|u(x; P, t_{\ell} z_0, t_{\ell} z)|: x\in \Gamma_R, z\in \Gamma_\rho\cup\Gamma, \ell=1,2\}
	\end{align*}
	for a fixed wavenumber $k$, a fixed $z_0\in \mathbb{R}^2\backslash \overline{B_\rho}$ and two positive constants $t_1, t_2$ with $t_1\neq t_2$, determine the obstacle-source pair $(\partial D, P)$.
\end{problem}

We refer to \cref{fig:setup} for an illustration of the geometry setting of \cref{prob:co-inversion}.

\begin{figure}
	\centering
	\begin{tikzpicture}[scale = .8, thick]
		\draw circle (5cm); 
		\draw node at (3.8, 3.8) {$\Gamma_R$};
		\draw node at (-2.6, 2.6) {$B_R$};
		\draw [blue] circle (2.65cm); 
		\draw node at (2.1, 2.1) {$\Gamma_\rho$};
		\draw node at (1.4, 1.2) {$B_\rho$};
		
		\draw[orange] (3, -3) to [out=180, in=195] (4, 0);
		\draw node at (2.2, -3) {$\Gamma$};
		\draw[orange, dashed] (4, 0) to [out=0, in=360] (3, -3);
		\draw node at (3.5, -1.5) {$G$};
		
		\fill [purple] (-2.5, -2.8) circle (0.06);
		\draw node at (-2.2, -3) {$z_0$};
		
		\pgfmathsetseed{40}	 
		\clip (0, 0) circle (2.65cm);
		\foreach \p in {1,...,20}
		{\fill [red] (3*rand, 3*rand) circle (0.06);
		}
		
		\pgfmathsetseed{21}			
		\draw[brown] plot[brown, smooth cycle, samples=8, domain={1:8}] (\x*360/7+2*rnd:0.1cm+1.6cm*rnd) [shading=ball, outer color=brown, inner color=white]; %fill=lightgray
		\draw node at (0, 0) {$D$};
		
	\end{tikzpicture}
	\caption{An illustration of the phaseless co-inversion problem. }\label{fig:setup}
\end{figure}
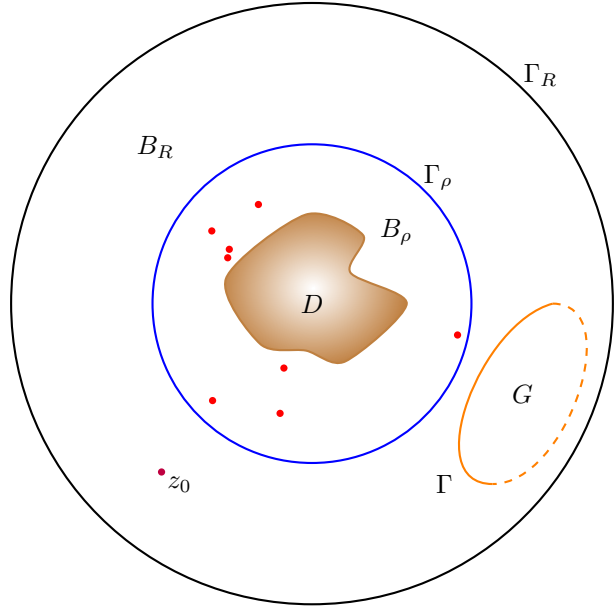

\section{Uniqueness}\label{sec: uniqueness}

In this section, we present the uniqueness result on \cref{prob:co-inversion}, which shows that the location of the point sources, the location and shape $\partial D$, as well as the boundary condition $\mathscr{B}$ for the obstacle, can be simultaneously and uniquely determined from the modulus of near-fields.

\begin{theorem}\label{Thm2.1}%-------------------------------------------------------------------------
	Let $D_1$ and $D_2$ be two obstacles with boundary conditions $\mathscr{B}_1$ and $\mathscr{B}_2$, and $P_1$ and $P_2$ be two sets of source points, such that $\cup_{j=1}^2 (D_j\cup P_j)\subset B_\rho\subset B_R$. Assume that $\Gamma$ is an admissible curve for $G$ with $G\subset B_R\backslash\overline{B_\rho}$. The scattered field and the total field with respect to $D_j \cup P_j$ and $D_j\cup \{z\}$ are denoted by $u_j^s(x; P_j)$, $u_j(x; P_j)$,  $u_j^s(x; z)$ and $u_j(x; z), j=1,2$, respectively. If the near-fields satisfy that
	\begin{align}
		|u_1(x; P_1)| & =|u_2(x; P_2)|, \quad \forall x\in \Gamma_R, \label{Thm2.1.1} \\
		|u_1(x; P_1, t_{\ell} z)| & =|u_2(x; P_2, t_{\ell} z)|, \quad \forall (x, z)\in \Gamma_R \times (\Gamma_\rho\cup\Gamma), \ \ell=1,2, \label{Thm2.1.2} \\
		|u_1(x; P_1, t_{\ell} z_0)| & =|u_2(x; P_2, t_{\ell} z_0)|, \quad \forall x\in \Gamma_R, \ \ell=1,2,  \label{Thm2.1.3} \\
		|u_1(x; P_1, t_{\ell} z_0, t_{\ell} z)| & =|u_2(x; P_2, t_{\ell} z_0, t_{\ell} z)|, \quad \forall (x, z)\in \Gamma_R \times (\Gamma_\rho\cup\Gamma), \ {\ell}=1,2,  \label{Thm2.1.4}
	\end{align}
	for an arbitrarily fixed wavenumber $k$, a fixed $z_0\in \mathbb{R}^2\backslash \overline{B_\rho}$ and two positive constants $t_1, t_2$ with $t_1\neq t_2$. Then we have  $D_1=D_2, \mathscr{B}_1 = \mathscr{B}_2$ and $P_1=P_2$.
\end{theorem}%-------------------------------------------------------------------------------------------

\begin{proof}
The proof is divided into two parts.

(i) {\bf Uniqueness for the obstacle.} From \cref{Thm2.1.1} and \cref{Thm2.1.2}, we have for all $x\in \Gamma_R, z\in \Gamma_\rho\cup\Gamma$ and $\ell=1, 2$,
$$
	t_{\ell}\left( | u_1(x; z)|^2-|u_2(x; z)|^2\right)= 2 \mathrm{Re}\left(u_2(x; P_2)\overline{u_2(x; z)}-u_1(x;P_1)\overline{u_1(x; z)}\right)
$$
where the overline denotes the complex conjugate. Further, by  $t_{\ell} >0 $ $(\ell=1,2)$ and $t_1\neq t_2$, we obtain
\begin{align}
	| u_1(x; z)| & =|u_2(x; z)|, \quad \forall (x, z)\in \Gamma_R \times (\Gamma_\rho\cup\Gamma), \label{Thm2.1.5} \\ 
	\mathrm{Re}\left(u_1(x; P_1)\overline{u_1(x; z)}\right) & =\mathrm{Re}\left(u_2(x; P_2)\overline{u_2(x; z)}\right), \ \forall (x, z)\in \Gamma_R \times (\Gamma_\rho\cup\Gamma). \label{Thm2.1.6}
\end{align}
Similarly, from \cref{Thm2.1.1}, \cref{Thm2.1.3} and \cref{Thm2.1.4}, we deduce that
\begin{align}
	|u_1(x; z_0)| & =|u_2(x; z_0)|, \quad \forall x\in \Gamma_R , \label{Thm2.1.7} \\
	|u_1(x; z, z_0)|&=|u_2(x; z, z_0)|, \quad \forall (x, z)\in \Gamma_R \times (\Gamma_\rho\cup\Gamma). \label{eq:equality_z_z0}
\end{align}

By using \cref{Thm2.1.5}, \cref{Thm2.1.7} and \cref{eq:equality_z_z0}, it can be seen that for all $x\in \Gamma_R, z\in \Gamma_\rho\cup\Gamma$,
$$
	\mathrm{Re}\left(u_1(x; z_0)\overline{u_1(x; z)}\right)=\mathrm{Re}\left(u_2(x; z_0)\overline{u_2(x; z)}\right).
$$
Then, following a similar argument of (12) and (13) in Theorem 2.2 in \cite{ZGSL20}, we know that
\begin{equation}\label{Thm2.1.8}
	u_1(x; z)=\mathrm{e}^{\mathrm{i}\gamma(x)}u_2(x; z),\quad \forall (x, z)\in \Gamma_R^0 \times (\Gamma_\rho^0\cup\Gamma^0)
\end{equation}
or
\begin{equation}\label{Thm2.1.9}
	u_1(x; z)=\mathrm{e}^{\mathrm{i}\eta(x)}\overline{u_2(x; z)},\quad \forall (x, z)\in \Gamma_R^0 \times (\Gamma_\rho^0\cup\Gamma^0),
\end{equation}
where $\gamma(x), \eta(x)$ are real-valued functions, and $\Gamma_R^0\subseteq \Gamma_R$, $\Gamma_\rho^0\subseteq \Gamma_\rho, \Gamma^0\subseteq\Gamma$ are open sets.

First, we consider the case \cref{Thm2.1.8}. From the reciprocity relation \cite[Theorem 3]{Athanasiadis} for point sources, we see
$$
\mathrm{e}^{\mathrm{i}\gamma(z)}u_2(z; x)=u_1(z; x)=u_1(x; z)=	\mathrm{e}^{\mathrm{i}\gamma(x)}u_2(x; z),\quad \forall (x, z)\in \Gamma_R^0 \times \Gamma^0_\rho,
$$
which yields $\gamma(x)=\gamma$, where $\gamma$ is a constant.
Then, the analyticity of $u_j(x; z) (j = 1, 2)$ with respect to $x$ leads to $u_1(x; z)=\mathrm{e}^{\mathrm{i}\gamma}u_2(x; z)$ for all $(x, z)\in \Gamma_R \times \Gamma^0_\rho$. Then, from the uniqueness of the obstacle scattering problem in $\mathbb{R}^2\backslash\overline{B_R}$ and the analyticity of $u_j(x; z) (j = 1, 2)$, we have
$$
u_1(x; z)=\mathrm{e}^{\mathrm{i}\gamma}u_2(x; z),\quad \forall x \in   \mathbb{R}^2\backslash (\overline{D_1\cup D_2}\cup \{z\}),
$$
i.e., for all $ x\in  \mathbb{R}^2\backslash (\overline{D_1\cup D_2}\cup \{z\})$,
$$
u_1^s(x; z)+\frac{\mathrm{i}}{4}H_0^{(1)}(k|x-z|)=\mathrm{e}^{\mathrm{i}\gamma}\left(u_2^s(x; z)+\frac{\mathrm{i}}{4}H_0^{(1)}(k|x-z|)\right).
$$

Now, by letting $x\to z$ and the boundedness of the scattered field $u_j^s(x; z) (j=1,2)$, we find $\mathrm{e}^{\mathrm{i}\gamma}=1$, and thus the far-field patterns coincide, i.e.
$$
u_1^\infty(\hat{x}; z)= u_2^\infty(\hat{x}; z),\quad  \forall (z, \hat{x})\in \Gamma_\rho^0 \times  \mathbb{S}^1,
$$
which, together with the mixed reciprocity relation \cite[Theorem 3.24]{CK19}, yields
$$
v_1^s(z; -\hat{x})=v_2^s(z; -\hat{x}),\quad  \forall  (z, \hat{x})\in \Gamma_\rho^0 \times  \mathbb{S}^1,
$$
where $v_j^s(x; d), x\in\mathbb{R}^2\backslash\overline{D_j}, d\in \mathbb{S}^1$, is the scattered field generated by the obstacle $D$ and the incident plane wave $v^i(x; d)=\mathrm{e}^{\mathrm{i} k x\cdot d}, j=1,2$. 
Further, from the analyticity of $v_j^s(x; d)(j = 1, 2)$ with respect to $x$, we have $v_1^s(x; d)=v_2^s(x; d)$ for all $x\in \Gamma_\rho, d \in \mathbb{S}^1$. And the uniqueness of the exterior scattering in $\mathbb{R}^2\backslash\overline{B_R}$ implies that
\begin{equation}\label{Thm2.1.10}
	v_1^\infty(\hat{x}; d)=v_2^\infty(\hat{x}; d),\quad \forall \hat{x}, d \in \mathbb{S}^1.
\end{equation}

Next, we are going to show that the case \cref{Thm2.1.9} does not hold. Suppose that \cref{Thm2.1.9} holds. From the reciprocity relation \cite[Theorem 3]{Athanasiadis} for point sources, we see
$$
u_1(z; x)=\mathrm{e}^{\mathrm{i}\eta(x)}\overline{u_2(z; x)},\quad \forall(x, z)\in \Gamma_R^0 \times \Gamma^0.
$$
Then, the analyticity of $u_j(z; x) (j = 1, 2)$ with respect to $z$ leads to $u_1(z; x)=\mathrm{e}^{\mathrm{i}\eta(x)}\overline{u_2(z; x)}$ for all $(x, z)\in \Gamma_R^0 \times \partial G$. Let $w(z; x)=u_1(z; x)-\mathrm{e}^{\mathrm{i}\eta(x)}\overline{u_2(z; x)}$, then
\begin{equation*}
	\begin{cases}
		\Delta w+ k^2 w = 0  & \text{in}\ G, \\
		\qquad\quad\ \ w =0 & \text{on}\ \partial G.
	\end{cases}
\end{equation*}
By the assumption of $G$ that $k^2$ is not a Dirichlet eigenvalue of $-\Delta$ in $G$, we obtain $w=0$ in $G$. Again, from the analyticity of $u_j(z; x)(j = 1, 2)$ with respect to $z$, it can be seen that for every $x\in \Gamma_R^0$,
$$
u_1(z; x)=\mathrm{e}^{\mathrm{i}\eta(x)}\overline{u_2(z; x)},\quad \forall z\in  \mathbb{R}^2\backslash (\overline{D_1\cup D_2}\cup \{x\}),
$$
i.e., for all $ z\in  \mathbb{R}^2\backslash(\overline{D_1\cup D_2}\cup \{x\})$,
$$
u_1^s(z; x)+\frac{\mathrm{i}}{4}H_0^{(1)}(k|x-z|)=\mathrm{e}^{\mathrm{i}\eta(x)}\left(\overline{u_2^s(z; x)}-\frac{\mathrm{i}}{4}\overline{H_0^{(1)}(k|x-z|)}\right).
$$
Now, by letting $z\to x$ and the boundedness of the scattered field $u_j^s(z; x) (j=1,2)$, we find $\mathrm{e}^{\mathrm{i}\gamma(x)}=1$,
and
$$
u_1(z; x)=\overline{u_2(z; x)},\quad \forall z\in  \mathbb{R}^2\backslash (\overline{D_1\cup D_2}\cup \{x\}).
$$
By taking $z=|z|\hat{z}$ and using the definition of far-field pattern (see \cite[Theorem 2.6]{CK19}), we obtain
\begin{align*}
	\lim\limits_{|z|\to\infty}|z|\mathrm{e}^{-\mathrm{i} k |z|}u_1(|z|\hat{z}; x) & =u_1^\infty(\hat{z}; x), \\
	\lim\limits_{|z|\to\infty}|z|\mathrm{e}^{\mathrm{i} k |z|}\overline{u_2(|z|\hat{z}; x)} & =\overline{u_2^\infty(\hat{z}; x)}.
\end{align*}
Noticing $u_1(|z|\hat{z}; x)=\overline{u_2(|z|\hat{z}; x)}$ and $u_1^\infty(\hat{z}; x)\neq 0$ for $\hat{z}\in S$ with some $S\subset \mathbb{S}^1$, we have
$$
\lim\limits_{|z|\to\infty}\mathrm{e}^{2\mathrm{i} k |z|}=\frac{\overline{u_2^\infty(\hat{z}; x)}}{ u_1^\infty(\hat{z}; x)},
$$
which is a contradiction. Hence, the case \cref{Thm2.1.9} does not hold.

Having verified \cref{Thm2.1.10}, by Theorem 5.6 in \cite{CK19}, we have $D_1=D_2$ and $\mathscr{B}_1 = \mathscr{B}_2$.

(ii) {\bf Uniqueness for the point sources.}

In the following, we will show that $P_1=P_2$. Since the obstacle and its boundary condition are uniquely determined, we know $u_1(x; z)=u_2(x; z)=: u(x; z)$ on $\Gamma_R$ for $z\in \Gamma_\rho\cup\Gamma$. From \cref{Thm2.1.6}, we have
\begin{equation}\label{Thm2.1.11}
	\mathrm{Re}\left(u(x; P_1)\overline{u(x; z)}\right)=\mathrm{Re}\left(u(x; P_2)\overline{u(x; z)}\right),\quad  \forall (x, z) \in \Gamma_R\times (\Gamma_\rho\cup\Gamma).
\end{equation}
According to \cref{Thm2.1.1}, we denote
$$
	u(x; P_j)=r(x)\mathrm{e}^{\mathrm{i} \alpha_j(x)},\quad u(x; z)=s(x, z)\mathrm{e}^{\mathrm{i} \beta(x, z)},
$$
where $r(x)=|u(x; P_j)|, s(x, z)=|u(x; z)|, \alpha_j(x)$ and $\beta(x, z)$ are real-valued functions, $j=1,2$.

Since $\Gamma$ is an admissible curve of $G$, by \Cref{def:admissible_curve}, the reciprocity relation \cite[Theorem 3]{Athanasiadis} and the analyticity of $u(z; x)$ with respect to $z$, we have $u(x; z)\not\equiv 0$ for $x\in \Gamma_R, z\in \Gamma$. Then, the continuity leads to $u(x; z)\neq 0, \forall x\in \Lambda_1, z\in \Sigma$, where $\Lambda_1\subseteq \Gamma_R$ and $\Sigma\subseteq \Gamma$ are open sets.  Similarly, by the analyticity and continuity of $u(x; P_j)(j=1,2)$ with respect to $x$, we deduce $u(x; P_j)\neq 0$ for $x\in \Lambda$ with an open set $\Lambda\subseteq \Lambda_1$. Therefore, $r(x)\neq 0, s(x, z)\neq 0, \forall (x, z)\in \Lambda\times \Sigma$.  Taking \cref{Thm2.1.11} into account, we derive that
$$
	\cos (\alpha_1(x)-\beta(x, z))=\cos(\alpha_2(x)-\beta(x, z)),\quad \forall (x, z)\in \Lambda\times \Sigma.
$$
Hence, either
\begin{equation}\label{Thm2.1.12}
	\alpha_1(x)=\alpha_2(x)+2m\pi,\quad \forall (x, z)\in \Lambda\times \Sigma
\end{equation}
or
\begin{equation}\label{Thm2.1.13}
	\alpha_1(x)=-\alpha_2(x)+2\beta(x, z)+2m\pi,\quad \forall (x, z)\in \Lambda\times \Sigma
\end{equation}
holds with some $m \in \mathbb{Z}$.

For case \cref{Thm2.1.12}, it is easy to see that for all $(x, z)\in \Lambda\times \Sigma$,
$$
	u(x; P_1)=r(x)\mathrm{e}^{\mathrm{i} \alpha_1(x)} =r(x)\mathrm{e}^{\mathrm{i} \alpha_2(x)+\mathrm{i}2m\pi}=u(x; P_2),
$$
which, together with the analyticity of $u(x; P_j)(j=1,2)$ and the uniqueness of the obstacle scattering problem in $\mathbb{R}^2\backslash \overline{B_R}$, yields
$$
	u(x; P_1)=u(x; P_2), \quad \forall x\in \mathbb{R}^2\backslash \overline{B_R}.
$$
Again, by the analyticity of $u(x; P_j), j=1,2$, we obtain
$$
	u(x; P_1)=u(x; P_2), \quad \forall x\in \mathbb{R}^2\backslash (\overline{D}\cup P_1\cup P_2),
$$
i.e. for all $x\in \mathbb{R}^2\backslash (\overline{D}\cup P_1\cup P_2)$,
$$
	u^s(x; P_1)+u^i(x; P_1)=u^s(x; P_2)+u^i(x; P_2).
$$

Suppose $P_1\neq P_2$. Then, there exists a source point $z^*$ such that $z^*\in P_1$ and $z^*\notin P_2$.
By the boundedness of $u^s(x; P_j)(j=1,2)$ and $z^*\notin P_2$ we have that $u(x; P_2)=u^s(x; P_2)+u^i(x; P_2)$ is bounded as $x\to z^*$. This is a contradiction since $z^*\in P_1$ and $u(x; P_1)=u^s(x; P_1)+u^i(x; P_1)$ is unbounded as $x\to z^*$. Hence, we have $P_1= P_2$.

Finally, we will show that the case \cref{Thm2.1.13} does not hold. Suppose that \cref{Thm2.1.13} is true, then we have for all $(x, z)\in \Lambda\times \Sigma$,
\begin{align*}
	u(x; P_1)\overline{u(x; z)}&=r(x)\mathrm{e}^{\mathrm{i} \alpha_1(x)} s(x, z)\mathrm{e}^{-\mathrm{i} \beta(x, z)}\\
	& = r(x)\mathrm{e}^{-\mathrm{i} \alpha_2(x)+\mathrm{i}2m\pi}s(x, z)\mathrm{e}^{\mathrm{i} \beta(x, z)}\\
	&=\overline{u(x; P_2)}u(x; z).
\end{align*}

By the reciprocity relation \cite[Theorem 3]{Athanasiadis} and $r(x)\neq 0, s(x, z)\neq 0, \forall (x, z)\in \Lambda\times \Sigma$, we deduce that for $x\in \Lambda$,
$$
	u(z; x)=\frac{u(x; P_1)\overline{u(z; x)}}{\overline{u(x; P_2)}},\quad \forall z\in \Sigma.
$$
Since $\Sigma\subseteq\Gamma$ and $\Gamma$ is an admissible curve of $G$, by \cref{def:admissible_curve} and the analyticity of $u(z; x)$ with respect to $z$, we obtain that
$$
	u(z; x)=\frac{u(x; P_1)\overline{u(z; x)}}{\overline{u(x; P_2)}},\quad \forall z\in \mathbb{R}^2\backslash (\overline{D}\cup \{x\}).
$$
Then, a contradiction can be derived by a similar discussion of case \cref{Thm2.1.9}. Hence the case \cref{Thm2.1.13} does not hold, which completes the proof.
\end{proof}

\begin{remark}
	We would like to point out that this artificial source method can also be applied to a co-inversion problem with phase information, and the corresponding theory of uniqueness can be established easily since the co-inversion problem can be decoupled linearly by the artificial source technique.
\end{remark}

%======================================

\section{Imaging algorithms}\label{sec:algorithms}

This section will decouple the phaseless co-inversion problem into two inverse scattering problems: a phaseless inverse obstacle scattering problem and an inverse source problem. Then the subproblems will be solved separately by the direct imaging methods.

\begin{problem}[Phaseless inverse obstacle scattering problem]\label{prob:obstacle1}
	Let $D$ be the impenetrable obstacle with boundary condition $\mathscr{B}$. Given the phaseless near-field data
	\begin{equation*}
		\{|u(x; P)|: x\in \Gamma_R\},  \quad
		\{|u(x; P, t_\ell z)|: x\in \Gamma_R, z\in \Gamma_\rho, \ell=1,2\},
	\end{equation*}
	for a fixed wavenumber $k$ and $t_1=\sigma, t_2=2\sigma, \sigma\ge 1$, reconstruct the location and shape of obstacle $D$.
\end{problem}

We will recover the phaseless data $|u(x; z)|$ for all $x\in \Gamma_R$, $z\in \Gamma_\rho$. From $t_1=\sigma$, $t_2=2\sigma$ and the measurements
\begin{equation*}
	\{|u(x; P)|: x\in \Gamma_R\},  \quad
	\{|u(x; P, t_\ell z)|: x\in \Gamma_R, z\in \Gamma_\rho, \ell=1,2\},
\end{equation*}
we have
\begin{align*}
	|u(x; \sigma z)|^2+ 2\mathrm{Re}(u(x; P)\overline{u(x; \sigma z)})+|u(x; P)|^2 & = |u(x; P, \sigma z)|^2, \\
	4 |u(x; \sigma z)|^2+ 4\mathrm{Re}(u(x; P)\overline{u(x; \sigma z)})+|u(x; P)|^2 & = |u(x; P, 2\sigma z)|^2,
\end{align*}
which implies for all $(x, z) \in \Gamma_R\times \Gamma_\rho$,
\begin{equation}\label{formular}
	|u(x; z)|=\frac{1}{\sqrt{2}\sigma}\sqrt{|u(x; P, 2\sigma z)|^2-2|u(x; P, \sigma z)|^2+|u(x; P)|^2}.
\end{equation}
Here, $\sigma$ is the scaling factor such that $\|u(\cdot; \sigma z)\|_{L^\infty(\Gamma_R)}\approx\|u(\cdot; P)\|_{L^\infty(\Gamma_R)}$ in \eqref{formular}.

With the phaseless data \eqref{formular}, we reconstruct the obstacle $D$ by using the reverse time migration (RTM) approach \cite{CH17}, which is based on the following indicator function
$$
	I_D(y)=-k^2\mathrm{Im}\int_{\Gamma_\rho}\int_{\Gamma_R}\Phi(y, z)\Phi(x, y)\Upsilon(x, z)\mathrm{d} s(x)\mathrm{d} s(z), \quad \forall y\in B_R,
$$
where
$$
	\Upsilon(x, z)=\frac{|u(x; z)|^2-|u^i(x; z)|^2}{u^i(x; z)},\quad \Phi(x, y)=\frac{\mathrm{i}}{4}H_0^{(1)}(k|x-y|).
$$

\begin{problem}[Phaseless inverse source problem]\label{prob:source}
	Given the phaseless near-field data
	$$
		\{|u(x; P)|: x\in \Gamma_R\},  \quad \{|u(x; P, t_\ell z)|: x\in \Gamma_R, z\in \Gamma_\rho, \ell=1,2\},
	$$
	for a fixed wavenumber $k$ and $t_1=\sigma$, $t_2=2\sigma$, determine the source points $P$.
\end{problem}

Let
\begin{align*}
	\Theta(x, z, P)=\frac{1}{\sigma}\left(2|u(x; P, \sigma z)|^2-\frac{1}{2}|u(x; P, 2\sigma z)|^2-\frac{3}{2}|u(x; P)|^2\right).
\end{align*}
We introduce the following indicator function for the reconstruction of $P$
$$
	I_P(y)=\frac{1}{\rho\sqrt{R}} \mathrm{Re}\int_{\Gamma_\rho}\int_{\Gamma_R}\mathrm{e}^{\mathrm{i}(k(|x-y|-|x-z|))}|x-z|^{\frac12}\Theta(x, z, P)\mathrm{d} s(x)\mathrm{d} s(z), \quad \forall y\in B_\rho\backslash \overline{D}.
$$

To illustrate the behavior of the indicator function, we assume that $R=c_0\rho$ with some $c_0>1$, and denote $d_0=\inf\limits_{z\in P,\, x\in\partial D}|z-x|$.

From \eqref{formular}, we see that
\begin{align*}
	\Theta(x, z, P) & =\frac{1}{\sigma}\left(|\sigma u(x; z)+u(x; P)|^2-|\sigma u(x; z)|^2 -|u(x; P)|^2\right)\\
	& =  \overline{u(x; P)}u(x; z)+  u(x; P)\overline{u(x; z)}.
\end{align*}
Therefore,
\begin{align*}
	I_P(y)&= \frac{1}{\rho\sqrt{R}}\ \mathrm{Re}\int_{\Gamma_\rho}\int_{\Gamma_R}\mathrm{e}^{\mathrm{i}k(|x-y|-|x-z|)}|x-z|^{\frac{1}{2}}\overline{u(x;P)}u(x; z)\mathrm{d}s(x)\mathrm{d}s(z)\\
	&\quad+  \frac{1}{\rho\sqrt{R}}\ \mathrm{Re}\int_{\Gamma_\rho}\int_{\Gamma_R}\mathrm{e}^{\mathrm{i}k(|x-y|-|x-z|)}|x-z|^{\frac{1}{2}}u(x;P)\overline{u(x; z)}\mathrm{d}s(x)\mathrm{d}s(z)\\
	&=\mathrm{Re}\ I_1(y)+\mathrm{Re}\ I_2(y).
\end{align*}

It is clear that
\begin{align*}
	I_1(y)&= \frac{1}{\rho\sqrt{R}}\int_{\Gamma_\rho}\int_{\Gamma_R}\mathrm{e}^{\mathrm{i}k(|x-y|-|x-z|)}|x-z|^{\frac{1}{2}}\overline{u^i(x; P)}u^i(x; z)\mathrm{d}s(x)\mathrm{d}s(z)\\
	&\quad+ \frac{1}{\rho\sqrt{R}}\int_{\Gamma_\rho}\int_{\Gamma_R}\mathrm{e}^{\mathrm{i}k(|x-y|-|x-z|)}|x-z|^{\frac{1}{2}}\overline{u^i(x;P)}u^s(x; z)\mathrm{d}s(x)\mathrm{d}s(z)\\
	&\quad+ \frac{1}{\rho\sqrt{R}}\int_{\Gamma_\rho}\int_{\Gamma_R}\mathrm{e}^{\mathrm{i}k(|x-y|-|x-z|)}|x-z|^{\frac{1}{2}}\overline{u^s(x;P)}u^i(x; z)\mathrm{d}s(x)\mathrm{d}s(z)\\
	&\quad+\frac{1}{\rho\sqrt{R}}\int_{\Gamma_\rho}\int_{\Gamma_R}\mathrm{e}^{\mathrm{i}k(|x-y|-|x-z|)}|x-z|^{\frac{1}{2}}\overline{u^s(x;P)}u^s(x; z)\mathrm{d}s(x)\mathrm{d}s(z)\\
	&\triangleq I_{11}+I_{12}+I_{13}+I_{14}.
\end{align*}

From
$$
H_0^{(1)}(t)=\sqrt{\frac{2}{\pi t}}\mathrm{e}^{\mathrm{i}(t-\frac{\pi}{4})}\left(1+\mathcal{O}\left(\frac{1}{t}\right)\right), \quad t\to\infty,
$$
we have
\begin{align*}
	&\quad\mathrm{e}^{\mathrm{i}(k|x-y|-k|x-z|)}|x-z|^{\frac{1}{2}}\overline{u^i(x; z_j)}u^i(x; z)\\
	&=\mathrm{e}^{\mathrm{i}(k|x-y|-\frac{\pi}{4})}\overline{u^i(x; z_j)} \ \mathrm{e}^{-\mathrm{i}(k|x-z|-\frac{\pi}{4})}|x-z|^{\frac{1}{2}}u^i(x; z) \\
	&=\mathrm{e}^{\mathrm{i}(k|x-y|-\frac{\pi}{4})}\overline{\frac{\mathrm{i}}{4}H_0^{(1)}(k|x-z_j|)}\ \mathrm{e}^{-\mathrm{i}(k|x-z|-\frac{\pi}{4})}|x-z|^{\frac{1}{2}} \frac{\mathrm{i}}{4}H_0^{(1)}(k|x-z|)\\
	&=\frac{1}{16}\mathrm{e}^{\mathrm{i} k(|x-y|-|x-z_j|)}\sqrt{\frac{2}{\pi k|x-z_j|}}\left(1+\mathcal{O}\left(\frac{1}{|x-z_j|}\right)\right)\sqrt{\frac{2}{\pi k}}\left(1+\mathcal{O}\left(\frac{1}{|x-z|}\right)\right)\\
	&=\frac{1}{8k\pi} \frac{\mathrm{e}^{\mathrm{i} k(|x-y|-|x-z_j|)}}{\sqrt{|x-z_j|}} 
	+\mathcal{O}\left(R^{-\frac{3}{2}}\right),
\end{align*}
and thus,
\begin{align}\label{estmate1}
	I_{11}(y)=\frac{1}{4k\sqrt{R}}\sum_{j=1}^{N} \int_{\Gamma_R}\frac{\mathrm{e}^{\mathrm{i} k(|x-y|-|x-z_j|)}}{\sqrt{|x-z_j|}}\mathrm{d} s(x) +\mathcal{O}\left({R^{-1}}\right).
\end{align}

From $|u^s(x; z)|=\mathcal{O}(R^{-1})$ and $|u^s(x; z_j)|=\mathcal{O}(R^{-\frac{1}{2}}d_0^{-\frac{1}{2}})(j=1,\cdots,N)$ for $x\in \Gamma_R$, $z\in \Gamma_\rho$, it can be seen that
\begin{align*}
	\overline{u^i(x; P)}u^s(x; z) & =\mathcal{O}\left(R^{-\frac{3}{2}}\right), \\
	\overline{u^s(x; P)}u^s(x; z) & =\mathcal{O}\left(R^{-\frac{3}{2}}d_0^{-\frac{1}{2}}\right),\\
	\overline{u^s(x;P)}u^i(x; z)  &  =\mathcal{O}\left(R^{-1}d_0^{-\frac{1}{2}}\right),
\end{align*}
which, together with \eqref{estmate1}, yields
\begin{equation} \label{estmate2}
	I_1(y)= \frac{1}{4k\sqrt{R}}\sum_{j=1}^N \int_{\Gamma_R}\frac{\mathrm{e}^{\mathrm{i} k(|x-y|-|x-z_j|)}}{\sqrt{|x-z_j|}}\mathrm{d} s(x)+\mathcal{O}\left(R^{-\frac{1}{2}}+d_0^{-\frac{1}{2}}\right).
\end{equation}
Similarly, we have 
\begin{equation} \label{estmate3}
	I_2(y)= \frac{1}{8k\pi \rho\sqrt{R}}\sum_{j=1}^N \int_{\Gamma_\rho}\int_{\Gamma_R}\frac{\mathrm{e}^{\mathrm{i} k(|x-y|+|x-z_j|-2|x-z|)}}{\sqrt{|x-z_j|}}\mathrm{d} s(x)\mathrm{d} s(z) +\mathcal{O}\left(R^{-\frac{1}{2}}+d_0^{-\frac{1}{2}}\right).
\end{equation}
Hence, 
\begin{align*} 
	I_{P}(y) & = \frac{1}{4k\sqrt{R}}\sum_{j=1}^N \int_{\Gamma_R}\frac{\cos (k|x-y|-k|x-z_j|)}{\sqrt{|x-z_j|}}\mathrm{d} s(x) \\
	&\quad+\frac{1}{8k\pi \rho\sqrt{R}}\sum_{j=1}^N \int_{\Gamma_\rho}\int_{\Gamma_R}\frac{\cos (k|x-y|+k|x-z_j|-2k|x-z|))}{\sqrt{|x-z_j|}}\mathrm{d} s(x)\mathrm{d} s(z) \\ 
	&\quad +\mathcal{O}\left(R^{-\frac{1}{2}}+d_0^{-\frac{1}{2}}\right).
\end{align*}

In virtue of the above proof, function $I_P(y)$ should decay as the sampling
point $y$ recedes from the corresponding source point $z_j$. And thus the source points $P$ can be recovered by locating the significant local maximizers of the indicator $I_P(y)$ over a suitable sampling region that covers $P$.

%========================================

\section{Numerical examples}\label{sec:examples}

In this section, we present several numerical examples to demonstrate that our approach is effective for the reconstruction of sources and obstacles from phaseless measurements. Synthetic forward data is computed by the boundary integral equation method. To test the stability of our co-inversion scheme, we add some random perturbations to the synthetic data such that 
$$
|u^\delta|=|u|+\delta r|u| 
$$
where $r$ is a uniform random number ranging from $-1$ to 1, and $\delta>0$ denotes the noise level. 

\begin{example}
In the first example, we consider the reconstruction of a sound-soft starfish-shaped obstacle whose boundary is given by the parametric form
$$
x(t)=(1+0.2\cos 5t)(\cos t, \sin t),\quad 0\le t<2\pi.
$$
In Figure \ref{fig: starfish}, we illustrate the reconstruction of the obstacle together with 4 sources located at $(3, 1), (2, 2), (-1.5, 3)$ and $(-2.5, -1.8)$. The wavenumber is fixed to be $k=4\pi$ and the scaling parameter $\sigma$ in Problem \eqref{prob:source} is chosen as $\sigma=1$. In Figure \ref{fig: starfish}(a) for the model setup, the 128 receivers (denoted by small blue stars) and 128 reference points (denoted by small black points) are uniformly deployed on the circle centered at the origin with radius 10 and 9, respectively. The sampling points for the sources forms a $200\times 200$ uniformly spaced grid over $[-5, 5]\times[-5, 5]$ while a $200\times 200$ uniformly spaced grid over $[-2, 2]\times[-2, 2]$ is used for imaging the obstacle. The normalized indicator functions for imaging the sources and the obstacle are depicted in Figure \ref{fig: starfish}(b)(d) and Figure \ref{fig: starfish}(c)(e), respectively. In Figure \ref{fig: starfish}(b)(d), the exact target sources are marked by the small red ``+'' and it can be seen that the recovered source locations match well with the significant local extreme values of the indicator function. In Figure \ref{fig: starfish}(c)(e), the exact boundary of the scatterer is marked by the black dashed line. All these results demonstrate that the inversion algorithm performs well in identifying the locations of the point sources as well as the shape of the obstacle, provided that the noise level is sufficiently small. 
\end{example}

\begin{figure}
\centering
\subfigure[]{\includegraphics[width=0.5\textwidth]{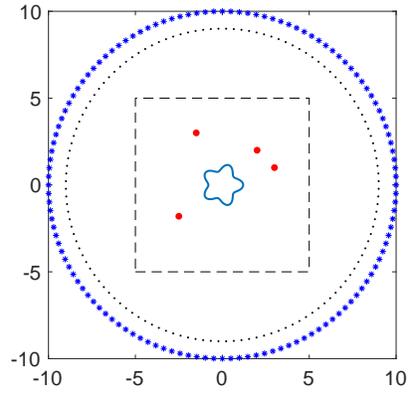}}\\
\subfigure[]{\includegraphics[width=0.4\textwidth]{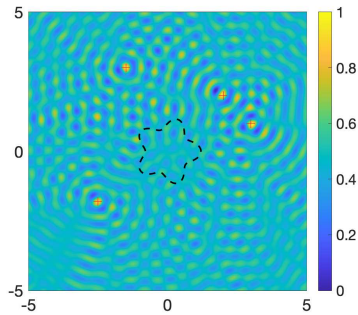}}\quad
\subfigure[]{\includegraphics[width=0.4\textwidth]{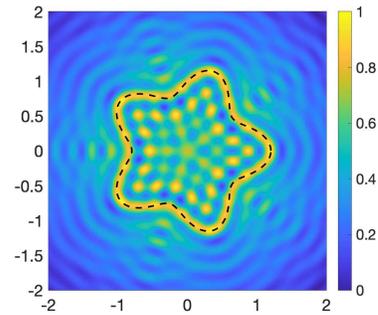}}\\
\subfigure[]{\includegraphics[width=0.4\textwidth]{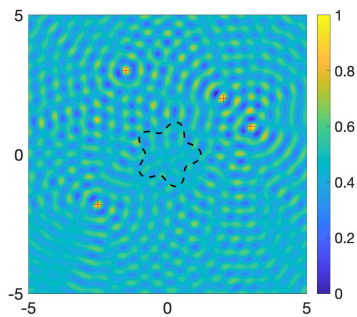}}\quad
\subfigure[]{\includegraphics[width=0.4\textwidth]{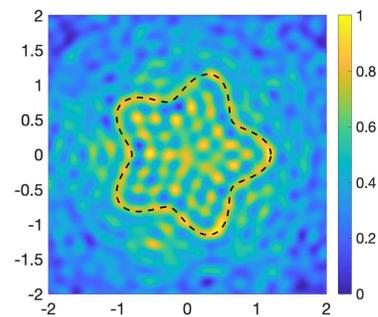}}
\caption{Reconstruction of a sound-soft obstacle and 4 sources. (a) Geometry setting of the model; (b) imaging of sources without noise; (c) imaging of obstacle without noise; (d) imaging of sources with 10\% noise; (e)  imaging of obstacle with 10\% noise.} \label{fig: starfish}
\end{figure}

\begin{example}
    In the second example, we test the inversion of a sound-hard peanut-shaped obstacle and 5 sources located at  $(3, 1), (2, 2), (-1.5, 3), (-2.5, -1.8)$ and $(1,-2.6)$. The exact boundary of the obstacle is given by
    $$
    x(t)=\sqrt{4\cos^2 t+\sin^2 t}(\cos t, \sin t),\quad 0\le t<2\pi.
    $$
    The wavenumber is chosen as $k=5\pi$ and the scaling parameter is still $\sigma=1$. In this example, 5\% noise is added to the synthetic data. Similar to the first example, a $200\times 200$ sampling grid is adopted to image the sources and the obstacle, respectively. We refer to Figure \ref{fig: peanut} for the illustration of the model setup concerning this example. We also compare the performance of the full aperture measurements and limited observations. For the full aperture case, 160 sensors and 160 artificial sources are utilized, whereas only half of the sensors and artificial sources are available in the limited aperture case. It can be seen from Figure \ref{fig: peanut} that, due to the lack of information, the non-illuminated portion of the targets (both source and obstacle) could not be well reconstructed. Meanwhile, the illuminated parts can still be satisfactorily recovered.
\end{example}

\begin{figure}
	\centering
	\subfigure[]{\includegraphics[width=0.4\textwidth]{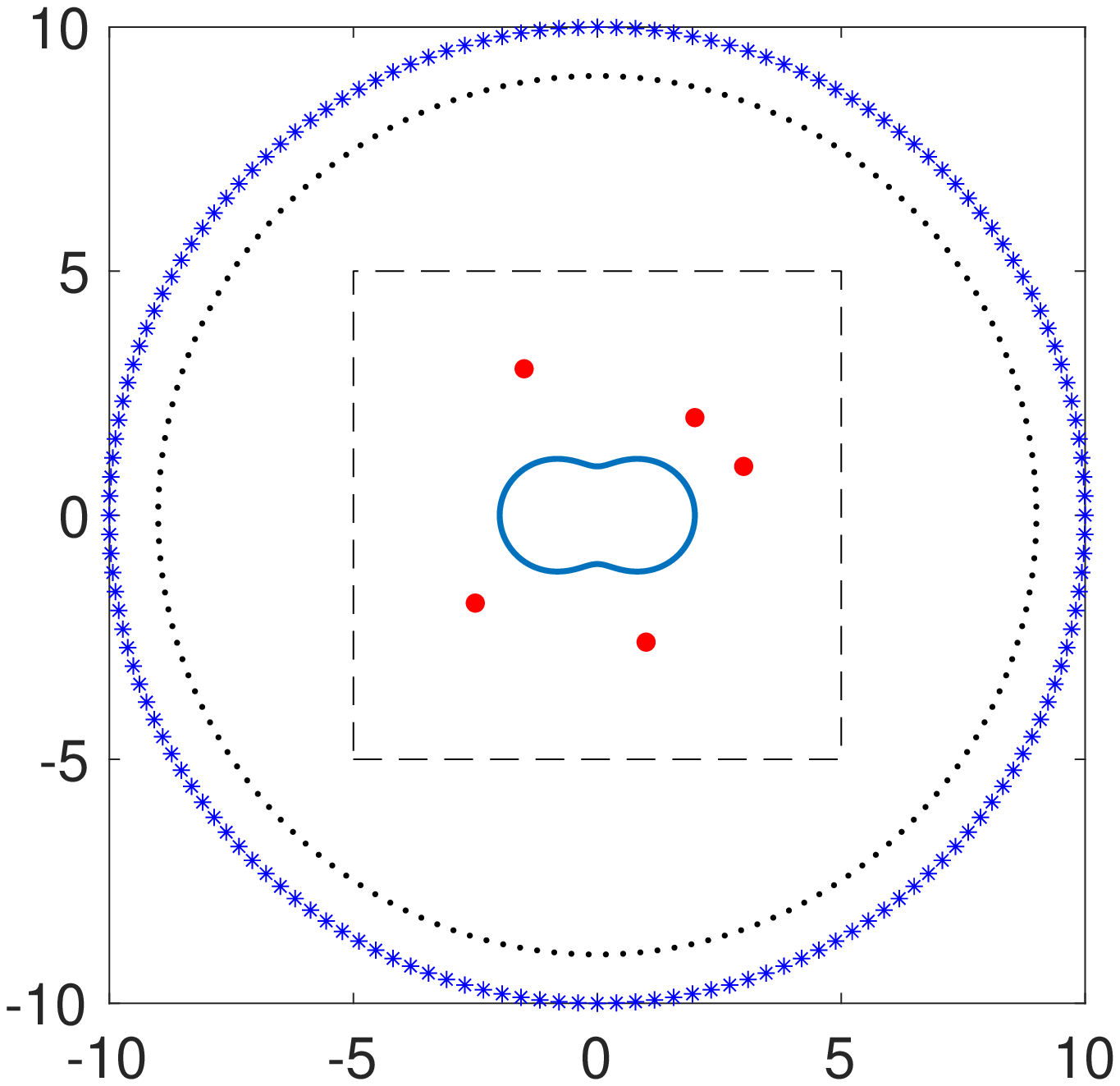}}\quad
	\subfigure[]{\includegraphics[width=0.4\textwidth]{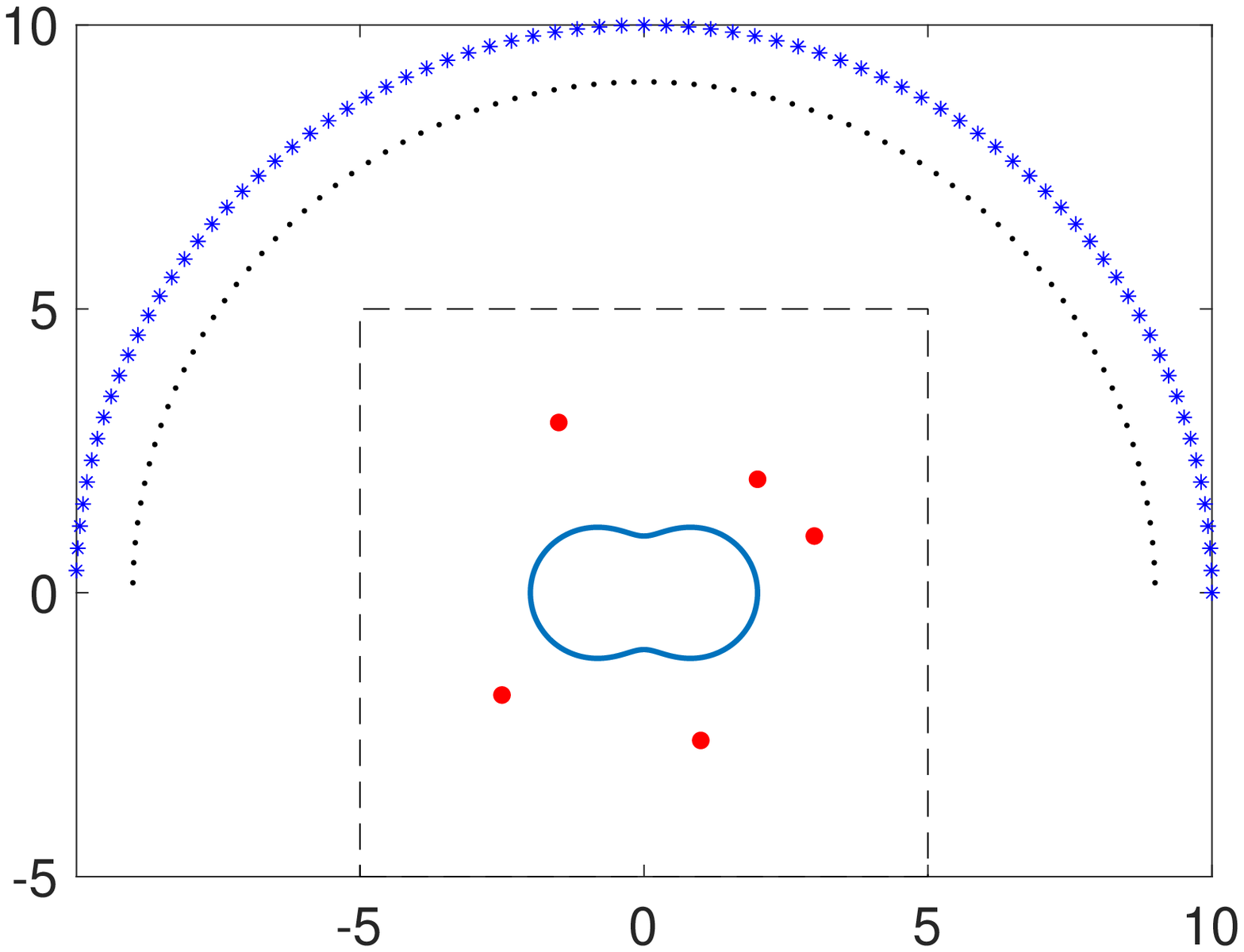}}\\
	\subfigure[]{\includegraphics[width=0.4\textwidth]{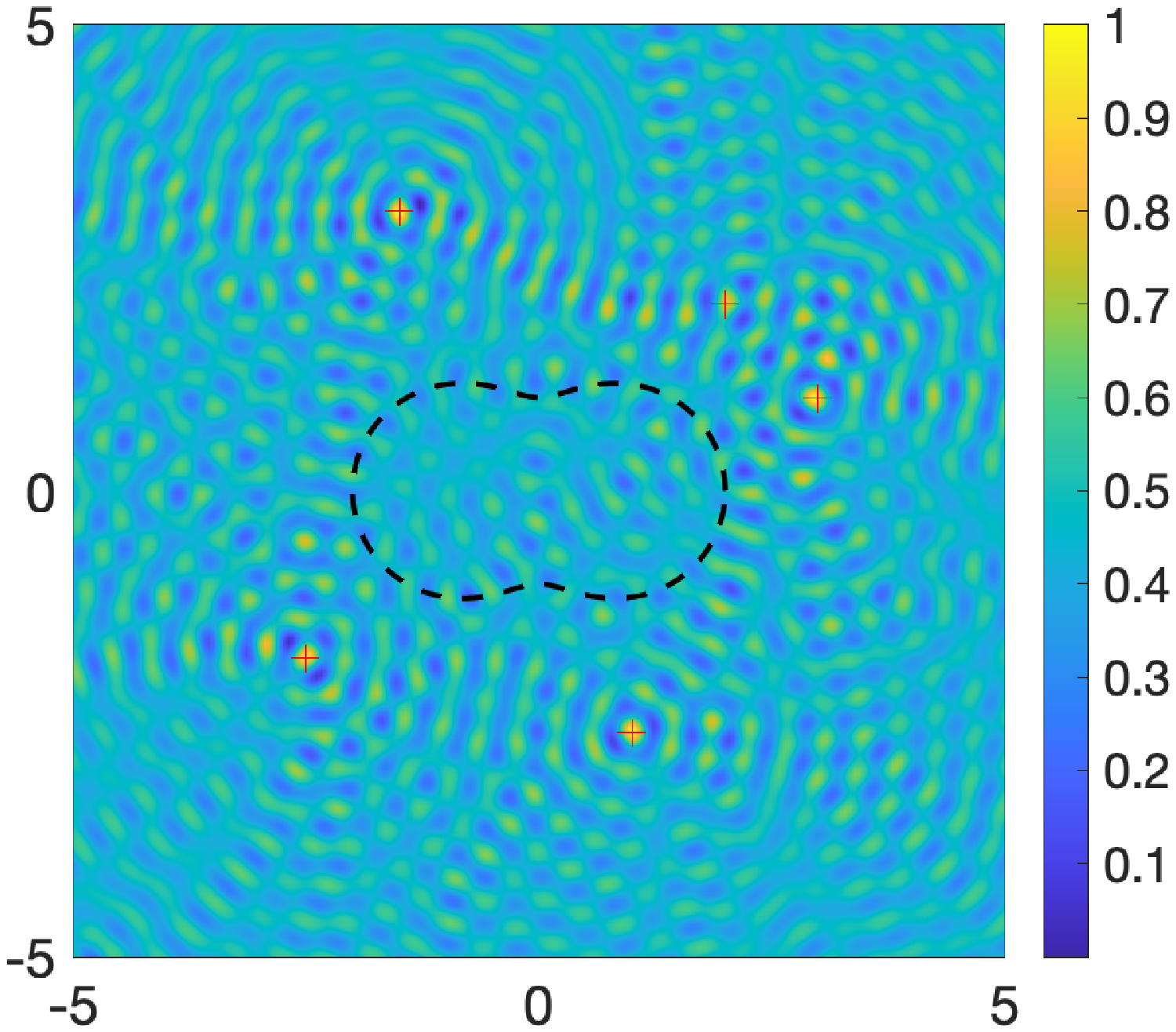}}\quad
	\subfigure[]{\includegraphics[width=0.4\textwidth]{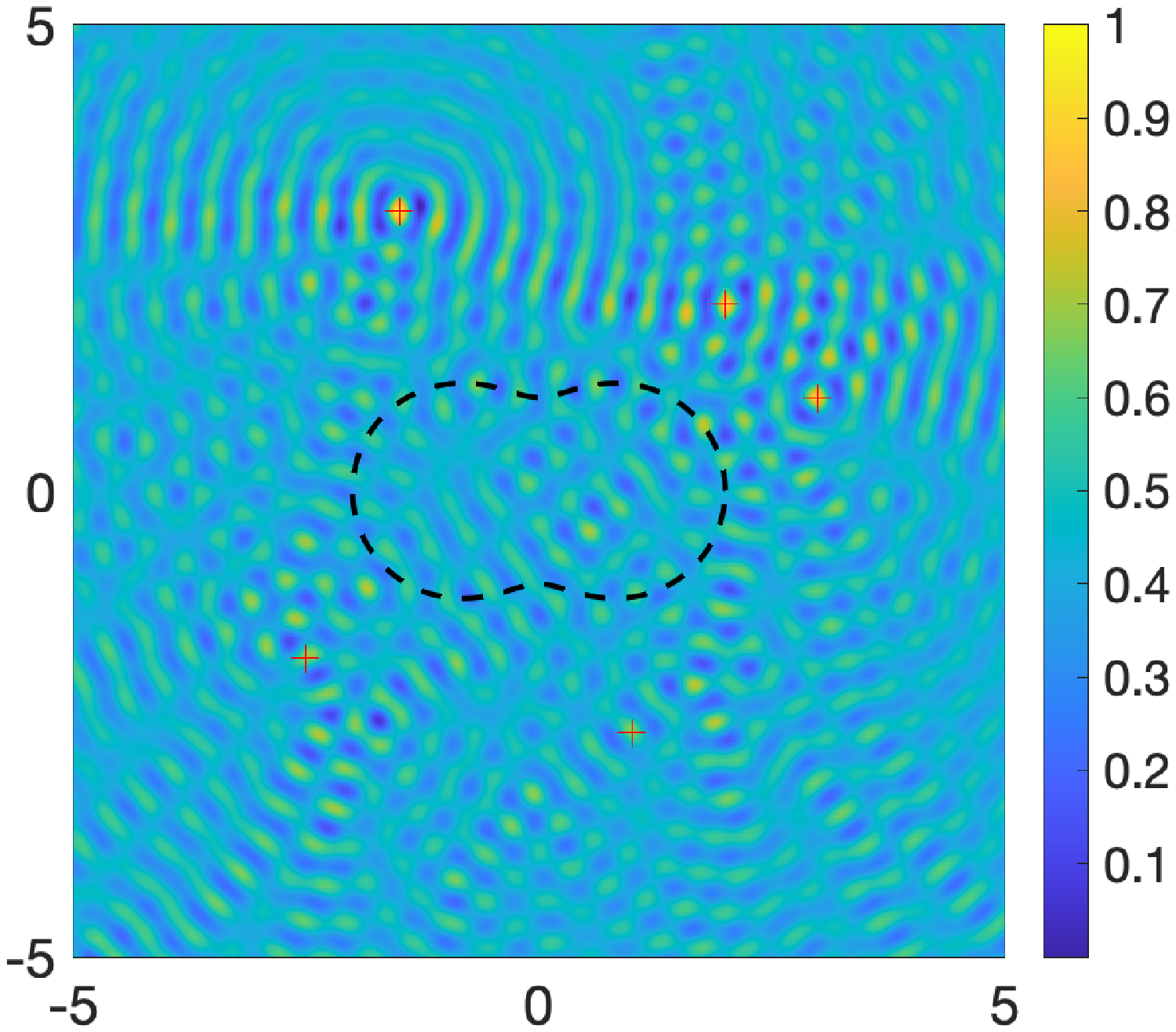}}\\
	\subfigure[]{\includegraphics[width=0.4\textwidth]{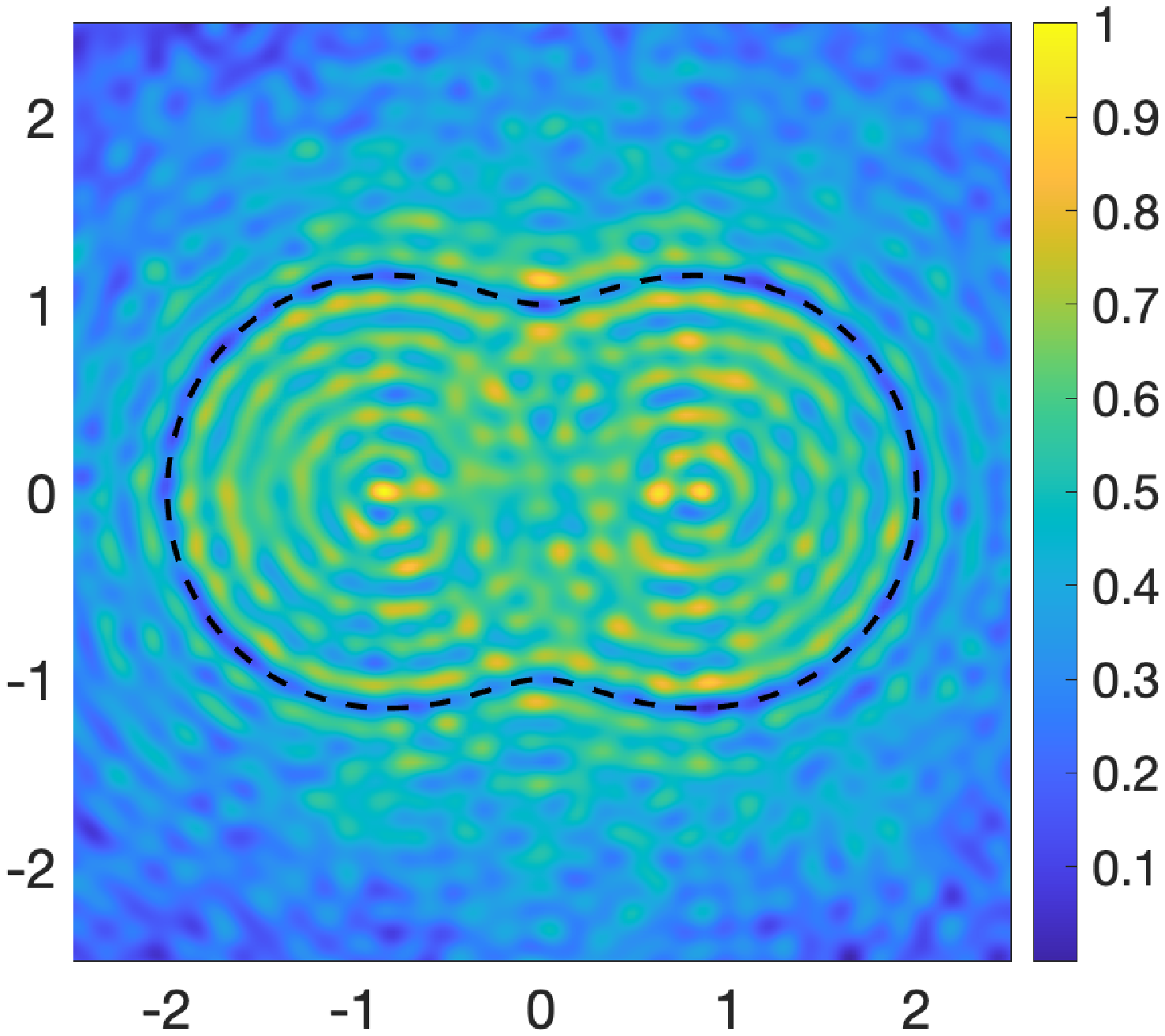}}\quad
	\subfigure[]{\includegraphics[width=0.4\textwidth]{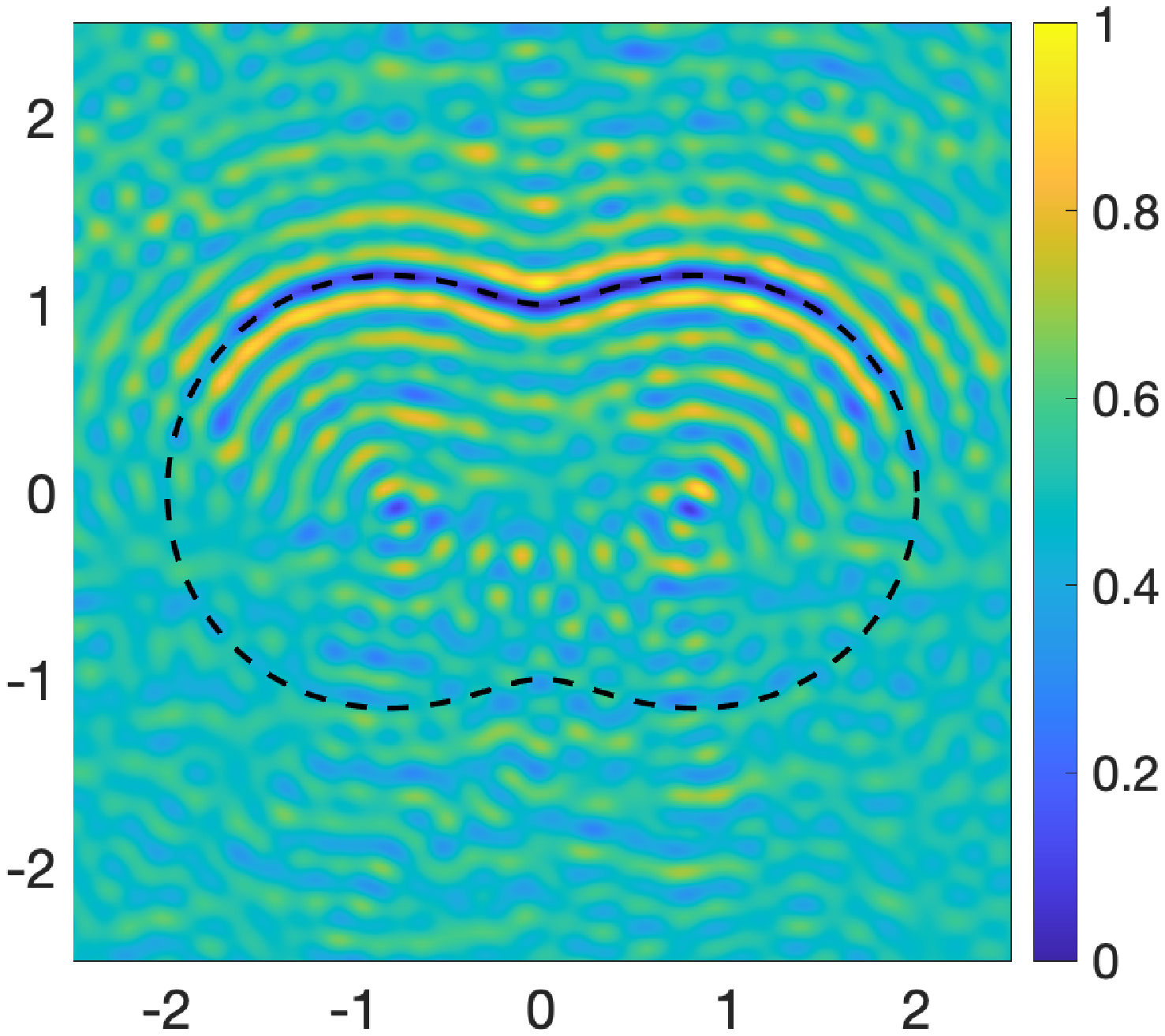}}
	\caption{Reconstruction of a sound-hard obstacle and 5 sources. Left column: full aperture data; Right column: limited aperture data; Top row:  geometry settings of the model; Middle row:  imaging of sources; Bottom row: imaging of obstacle.} \label{fig: peanut}
\end{figure}

\begin{example}\label{exa:kite}
	In the last example, the co-inversion of two sound-soft obstacles and the corresponding source points is considered. The scatterer consists of a kite-shaped obstacle
	$$
	x(t)=(-2+\cos t+0.65\cos 2t-0.65, 2+1.5\sin t), \quad 0\le t<2\pi.
	$$
	and a unit disk centered at $(3, -3)$. We set $k=18, \sigma=1$ and $\delta=5\%$ in this example. The receivers and reference sources are located uniformly on the circle with radii 15 and 14, respectively. For the reconstruction, here a $200\times 200$ sampling grid over $[-6, 6]\times[-6, 6]$ is utilized to image the sources and obstacles. The reconstructions are demonstrated in Figure \ref{fig: kite}. One could easily observe that the quality of obstacle recovery deteriorates greatly as the number of measurements decreases. At first glance, it seems from Figure \ref{fig: kite}(c)(d) that the reconstruction of sources is not significantly affected by the reduction of data. However, if we collect the coordinates of local maximizers over the sampling grid and compare them with the exact source points,  then it can be immediately seen that the case of insufficient data would produce a mismatched reconstruction of one source point. This comparison is illustrated in Figure \ref{fig:kite_source} where the exact target sources are marked by the small red ``+'' and the reconstructions are marked by the small black ``$\circ$''. These results show that a diminution of the data could lead to incorrect reconstructions, hence the availability of adequate data is essential for the accuracy of recoveries.
\end{example}

\begin{figure}
	\centering
	\subfigure[]{\includegraphics[width=0.45\textwidth]{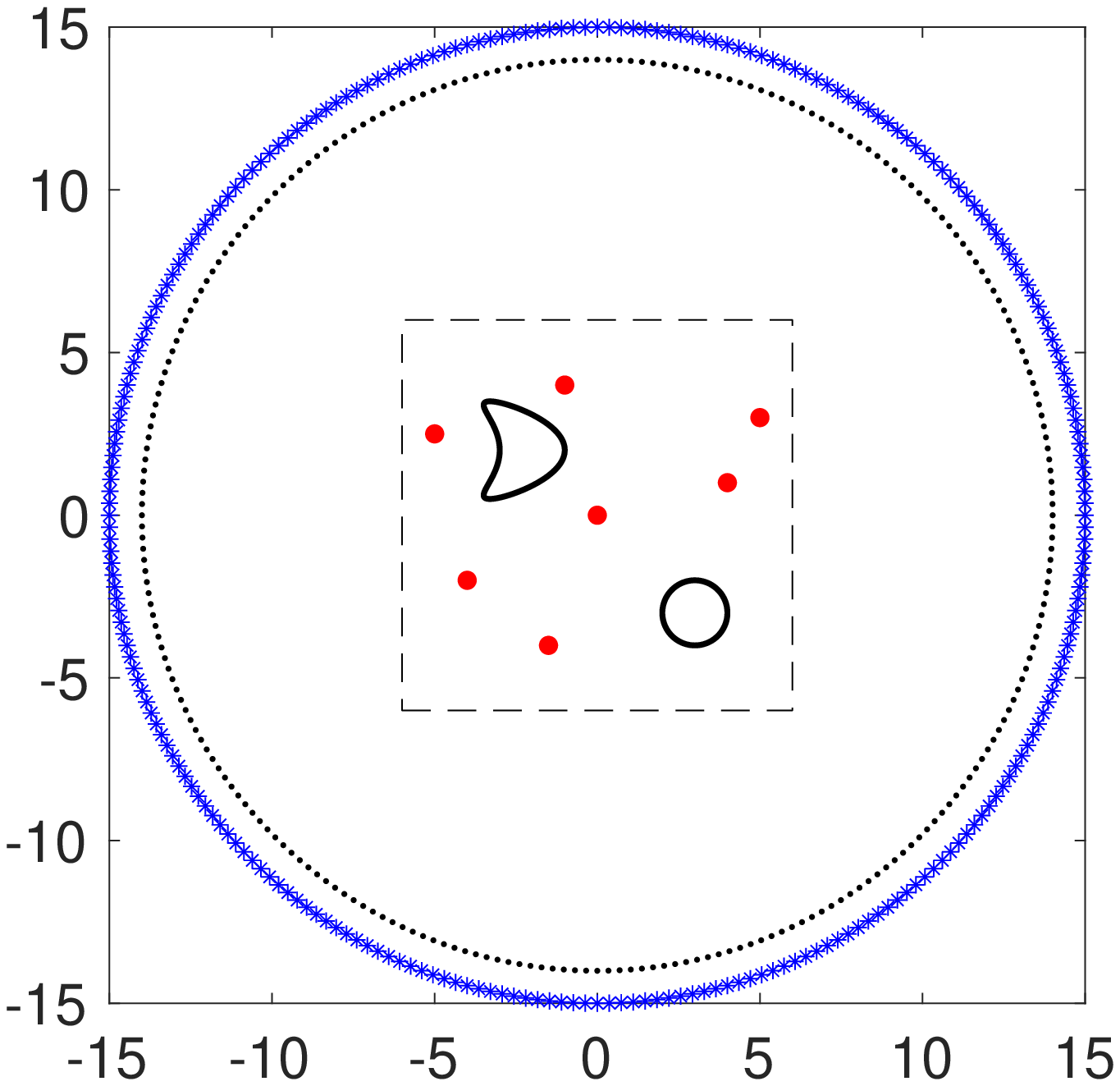}}\quad
	\subfigure[]{\includegraphics[width=0.45\textwidth]{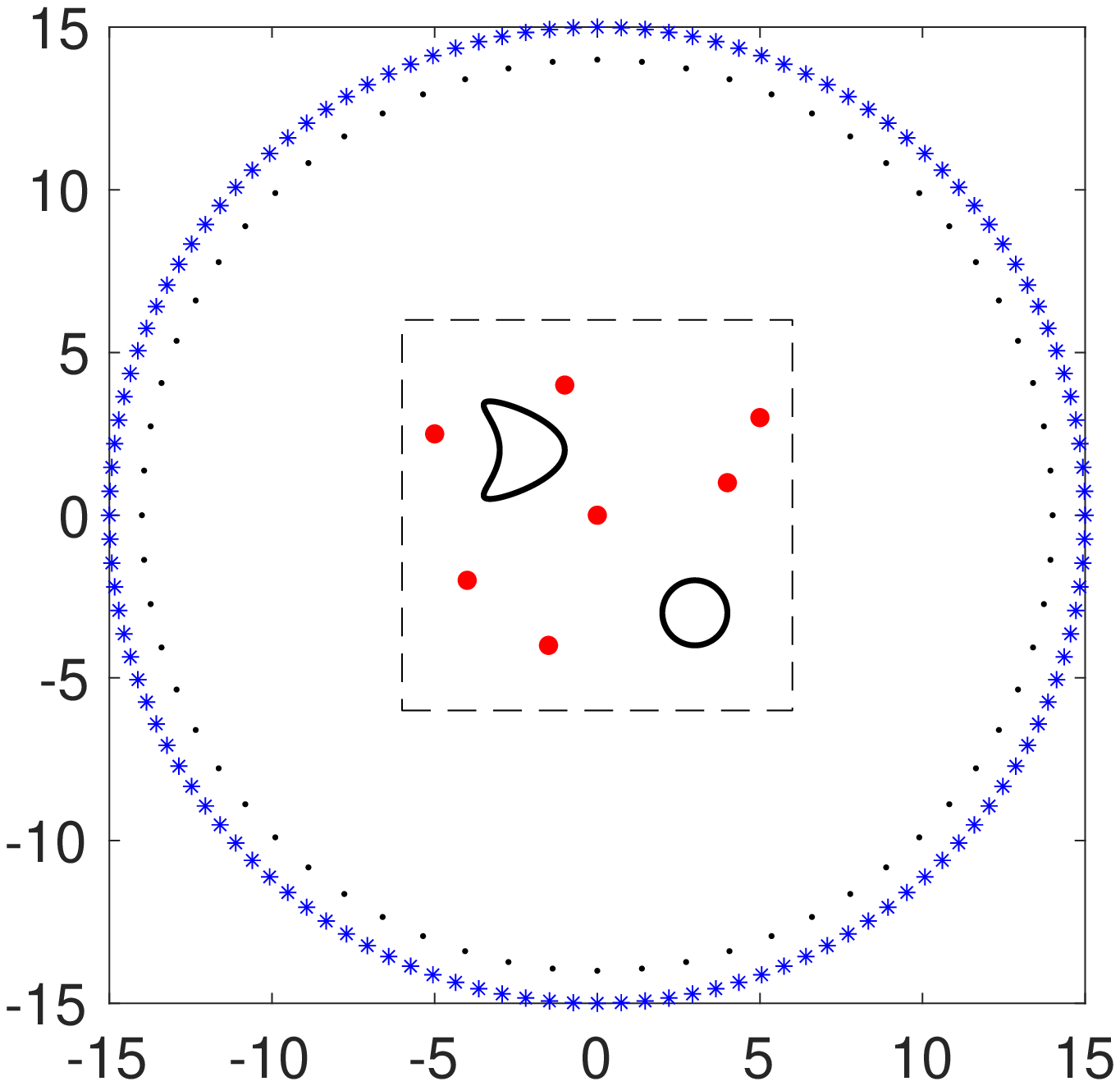}}\\
	\subfigure[]{\includegraphics[width=0.45\textwidth]{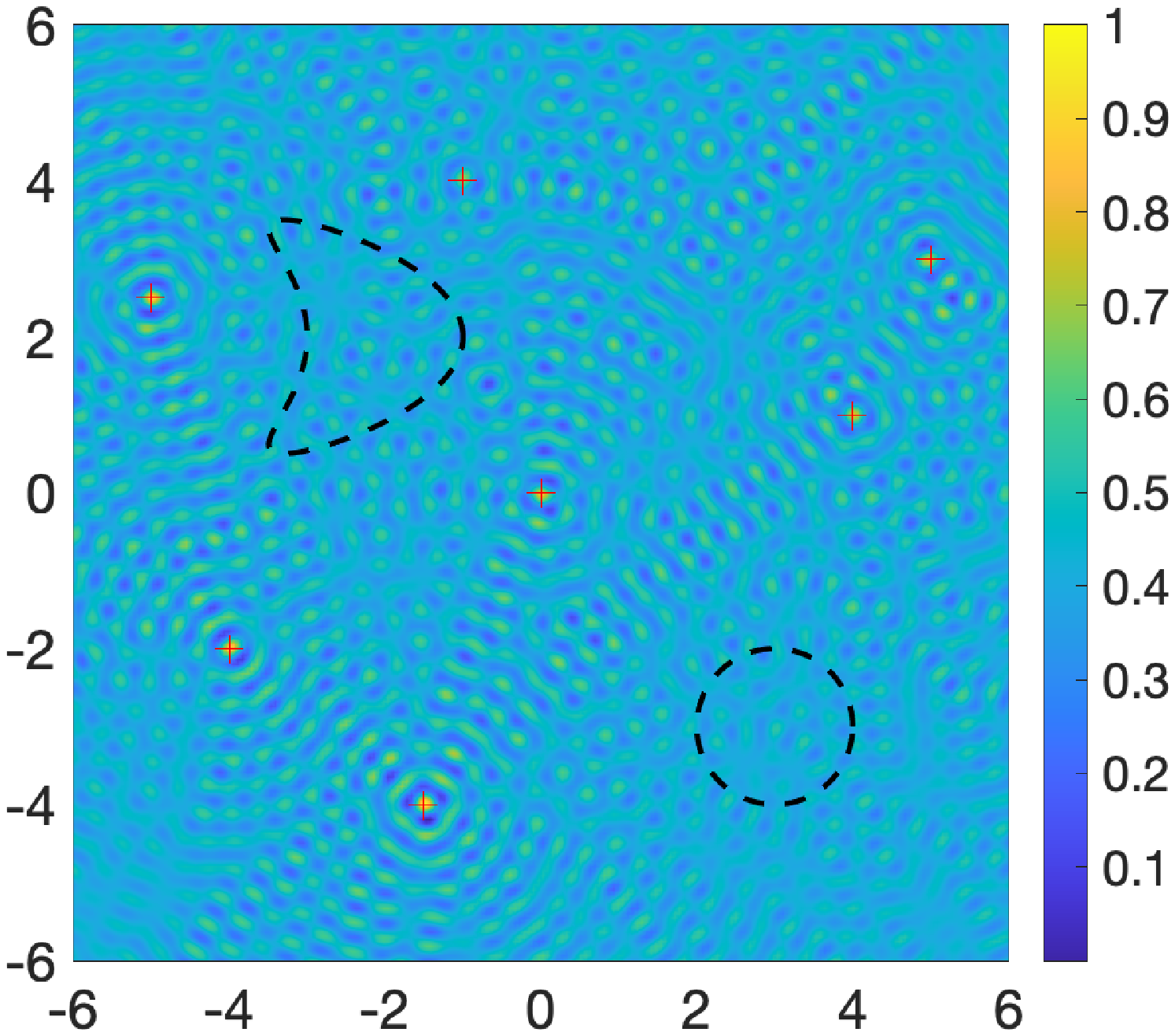}}\quad
	\subfigure[]{\includegraphics[width=0.45\textwidth]{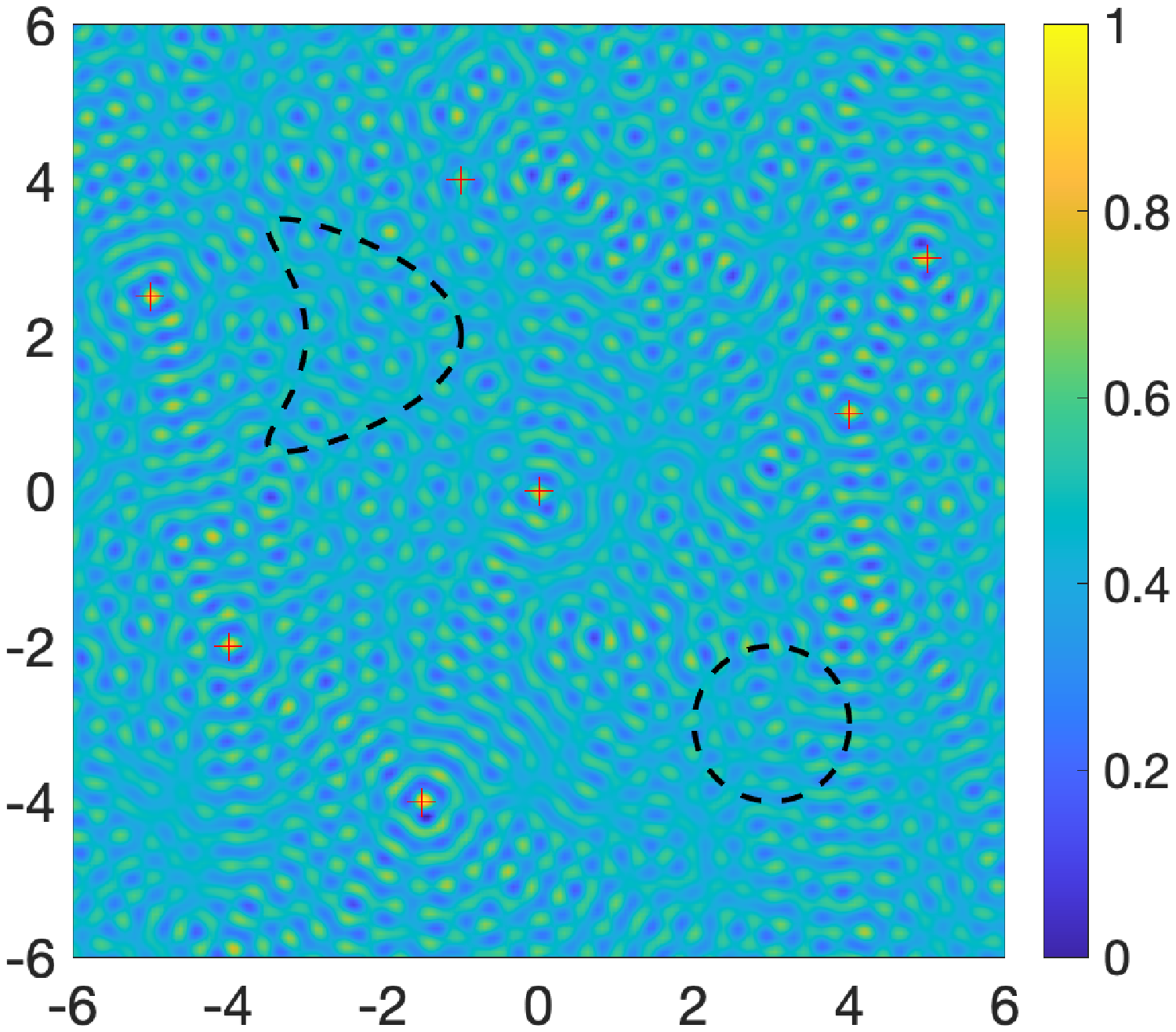}}\\
	\subfigure[]{\includegraphics[width=0.45\textwidth]{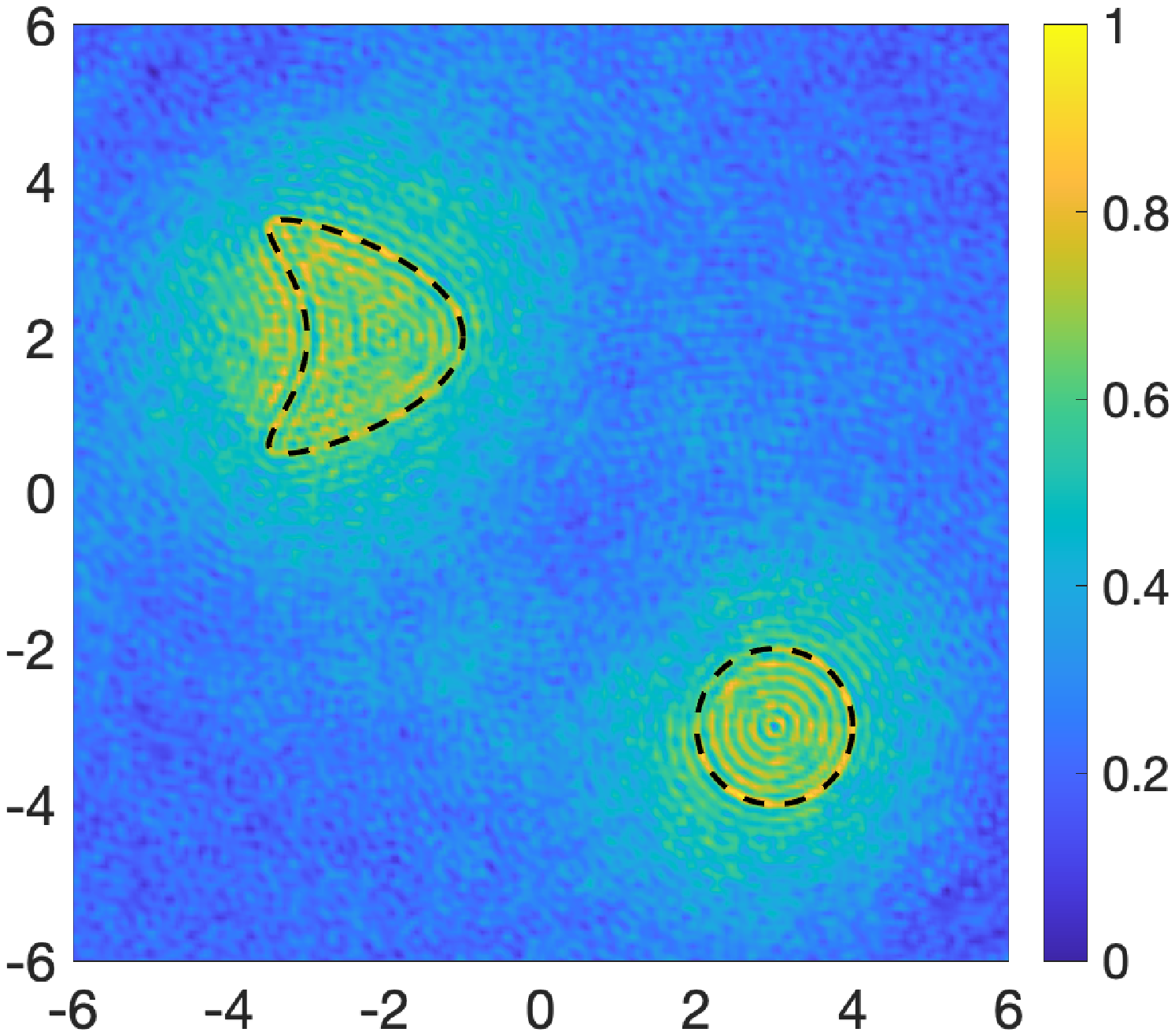}}\quad
	\subfigure[]{\includegraphics[width=0.45\textwidth]{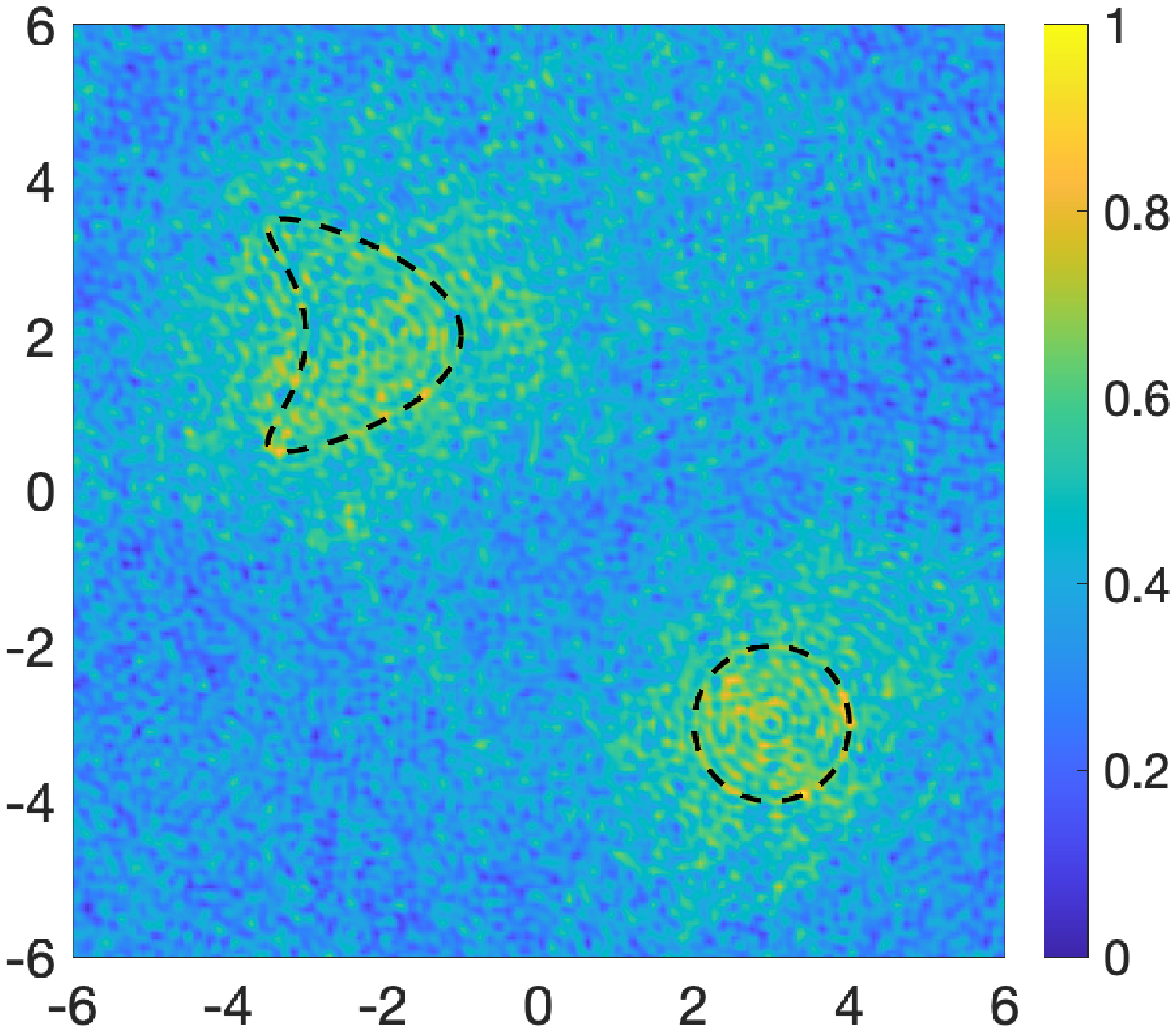}}
	\caption{Reconstruction of two obstacles and 7 sources. Left column: 256 receivers and 256 reference sources; Right column: 128 receivers and 64 reference sources; Top row:  geometry settings of the model; Middle row:  imaging of sources; Bottom row: imaging of obstacles.} \label{fig: kite}
\end{figure}

\begin{figure}
	\centering
	\subfigure[]{\includegraphics[width=0.45\textwidth]{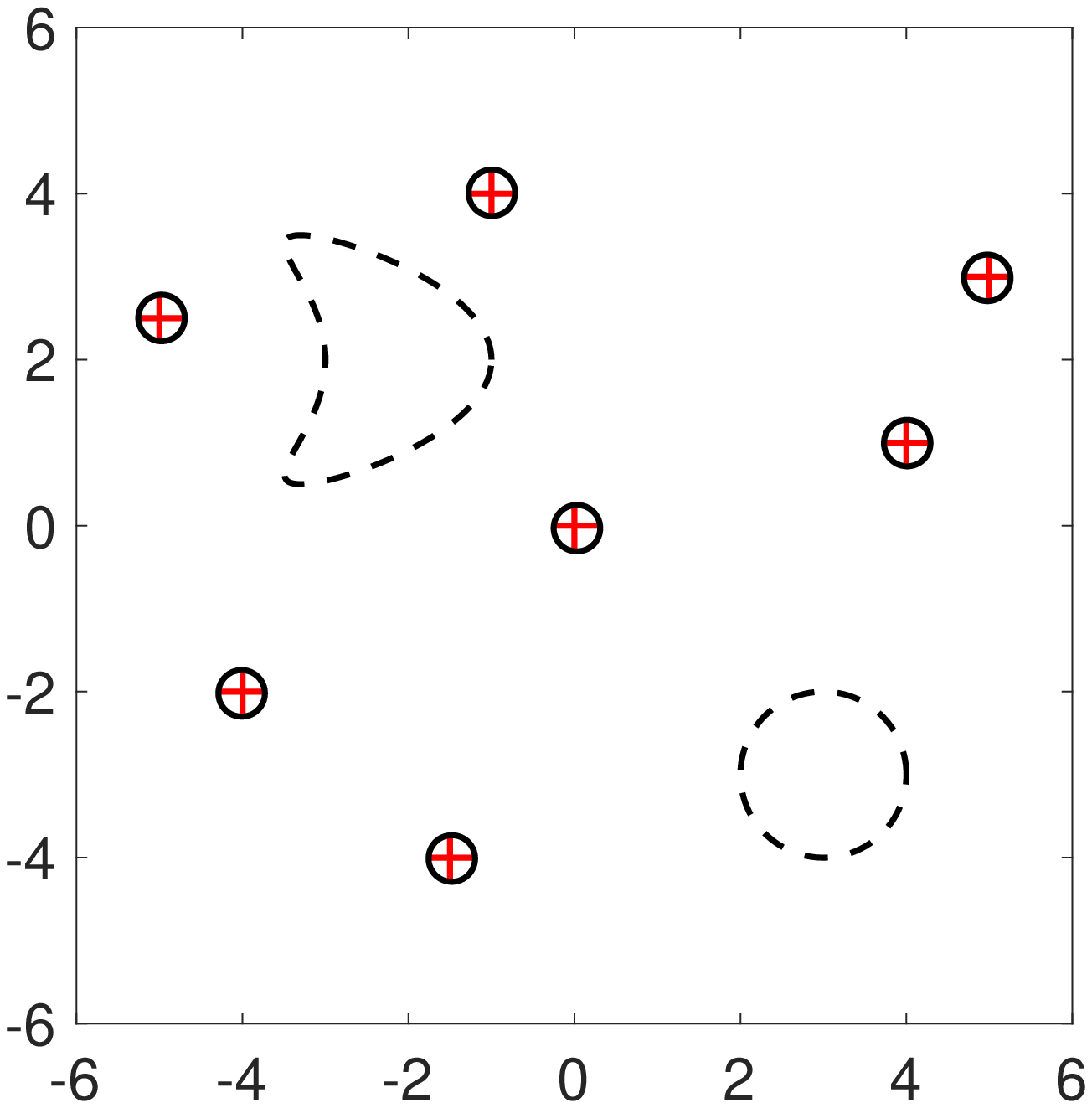}}\quad
	\subfigure[]{\includegraphics[width=0.45\textwidth]{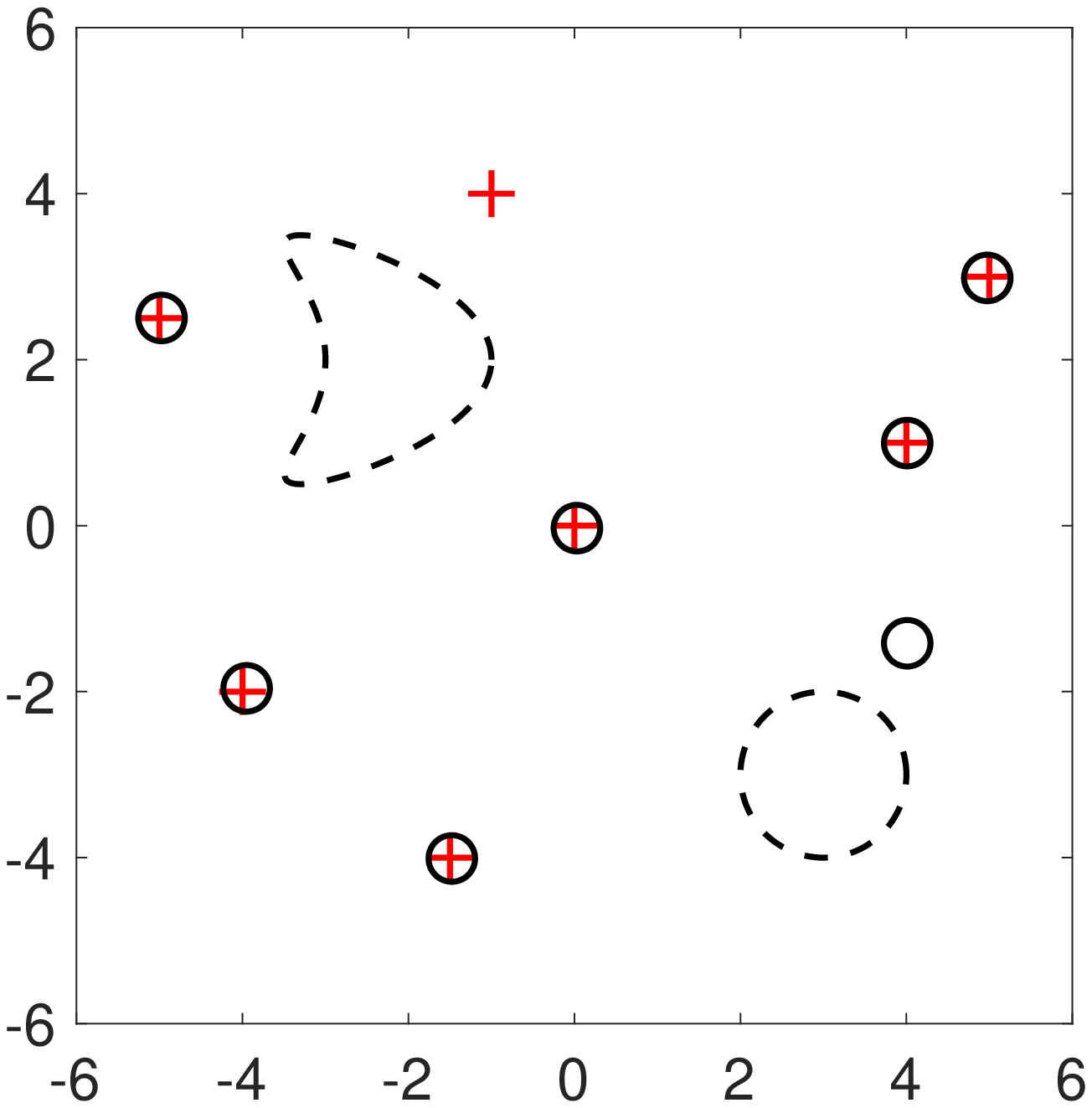}}
	\caption{A comparison of the source reconstruction. (a) 256 receivers and 256 reference sources; (b) 128 receivers and 64 reference sources.} \label{fig:kite_source}
\end{figure}

\begin{table}
	\centering
	\caption{Reconstruction of source locations in Example \ref{exa:kite}.}\label{tab:location}
	\begin{tabular}{cccc}
		\toprule
		        & Exact locations & Reconstruction (more data) &  Reconstruction (less data)\\ \midrule
		Point 1 &    $(0, 0)$     &   $(0.0302, 0.0302 )$  &   $(0.0302, 0.0302 )$   \\
		Point 2 &  $(-1.5, -4)$   &   $(-1.4774, -4.0101)$    &   $(-1.4774, -4.0101)$  \\
		Point 3 &    $(5, 3)$     &   $(4.9749, 2.9849)$    &   $(4.9749, 2.9849)$ \\
		Point 4 &   $(-5, 2.5)$   &   $(-4.9749, 2.5025)$   &   $(-4.9749, 2.5025)$ \\
		Point 5 &    $(-1, 4)$    &   $(-0.9950, 4.0101)$   &  $(4.0101, -1.4171)$ \\
		Point 6 &   $(-4, -2)$    &   $(-3.9497, -1.9598)$    &   $(-3.9497, -1.9598)$  \\
		Point 7 &    $(4, 1)$     &   $(4.0101, 0.9950)$    &   $(4.0101, 0.9950)$ \\ \bottomrule
	\end{tabular}
\end{table}

\section*{Acknowledgments}

D. Zhang and Y. Wu were supported by NSFC grant 12171200 and Y. Guo was supported by NSFC grant 11971133.

\end{document}